\documentclass{amsart}
\usepackage{graphicx}
\usepackage{dsfont}
\usepackage{color}
\usepackage{bbold}
\usepackage{mathtools}
\mathtoolsset{showonlyrefs}

\author[L.~Boulton]{Lyonell Boulton$^\beta$} 
\address{$^\beta$Department of Mathematics and Maxwell Institute for Mathematical Sciences, Heriot-Watt University, Edinburgh, EH14 4AS.}
\author[M.~Marletta]{Marco Marletta$^\mu$}
\address{$^\mu$School of Mathematics, Cardiff University, Senghennydd Road, Cathays, Cardiff, CF24 4AG.}
\date{June 2024}
\title{The pseudospectrum of an operator with Bessel-type singularities}
\newcommand{\sgn}{\operatorname{sgn}}
\newcommand{\quasi}{\mathrm{u}}

   \newtheorem{theorem}{Theorem}
     \newtheorem{lemma}{Lemma}
     \newtheorem{proposition}{Proposition}

     \newtheorem{remark}{Remark}

\begin{document}



\begin{abstract}
In this paper we examine the asymptotic structure of the pseudospectrum of the singular Sturm-Liouville operator $L=\partial_x(f\partial_x)+\partial_x$ subject to periodic boundary conditions on a symmetric interval, where the coefficient $f$ is a regular odd function that has only a simple zero at the origin. The operator $L$ is closely related to a remarkable model examined by Davies in 2007, which exhibits surprising spectral properties balancing symmetries and strong non-self-adjoitness. In our main result, we derive a concrete construction of classical pseudo-modes for $L$ and give explicit exponential bounds of growth for the resolvent norm in rays away from the spectrum. 

This paper is dedicated to Professor E. Brian Davies FRS on the occasion of his 80th birthday and the final version appears in the 
Journal of Spectral Theory, {\tt https://ems.press/journals/jst/articles/14297856}.
\end{abstract}

\maketitle


\section{Introduction}

Let $f\in \mathrm{C}^2([-1,1];\mathbb{R})$ 
be an odd function such that $f'(0)\not=0$ and $f(x)>0$ for $x>0$.
We associate with $f$ a second order differential expression $l$ defined on suitable functions $u$ by
\[
    (l u)(x)=(f(x)u'(x)+u(x))'.
\]
Note that $l$ has an interior singularity at $x=0$.

In this paper we shall study the pseudospectra of a closed operator realisation $L$ of $l$, specified by
periodic boundary conditions and regularity at $x=0$. We focus particularly on the asymptotic behaviour of the resolvent norm far from the spectrum. The fact that the numerical range of $L$ is easily
shown to be the whole complex plane means that there is no `direction of escape' to infinity in which the
behaviour of the pseudospectra can easily be foreseen.

Interest in this type of operator seems to have started with an idealised model for a thin layer of viscous fluid inside a rotating cylinder,
examined by Benilov, O'Brien and Sazonov \cite{BOS2003}; their case corresponds to $f(x)=\frac{2\varepsilon}{\pi}\sin(\pi x)$ ($0<\varepsilon<1$),
resulting in additional singularities at the endpoints. Benilov \emph{et al.} presented various observations concerning the spectrum and the associated time-evolution linear problem - most prominently, they noted that the spectrum appeared to be purely imaginary. 
Their conjectures have been systematically addressed in recent years for various trigonometric $f$, resulting in several articles concerning both
the spectrum and the ill-posedness of the corresponding evolution equation, e.g. \cite{Davies2007, ChugPelin2009, ChugKarabas, ChugVolk, Weir2009, Weir2010, DaviesWeir2010}. Some of these results were generalised  in \cite{BLM2010} to a wider class of $f$ (odd, positive on $(0,\pi)$ and $2\pi$-periodic, with various regularity assumptions), where it is shown that the eigenvalues are always purely imaginary; and also in \cite{BMR2012}, where it is shown that the spectrum is discrete and infinite, the resolvent is of Schatten class ${\mathcal C}_p$ for $p>\frac{2}{3}$, but the eigenfunctions do not form an unconditional basis.

Beyond these results, for $f(x)=\frac{2\varepsilon}{\pi} \sin(\pi x)$, Figure~6 of \cite{BOS2003} shows a numerically computed graph of the pseudospectrum, rotated by $\frac{\pi}{2}$, for one (unspecified) value of $\varepsilon$. The pseudospectral level sets appear qualitatively to be curves asymptotically close to parabolas fitting the symmetry of the spectrum, though to the best of our knowledge this has never been proved.
Similar pictures appear for a related operator on pages 124--125 and 406--408 of \cite{TE2005}, again obtained by numerical methods.

Most rigorous results that we know concerning pseudospectra of differential operators such as $L$, depend upon the construction of pseudo-modes. 
We shall give a review of some of these immediately after the statement of Theorem  \ref{th:thepseudospectrum} below, but broadly they are either semi-classical \cite{DSZ} or an evolved form of WKB-type \cite{KS2017}.
The pseudo-modes that they generate are localised functions, and no boundary 
conditions need be accommodated. By contrast, below we construct pseudo-modes for the operator $L$  by modifying the unique regular-at-the-origin solution of the differential equation $l \phi = E\phi$ using a `periodiser',
to satisfy periodic boundary conditions. We analyse the behaviour of these resulting pseudo-modes as the spectral parameter $E$ escapes to infinity on rays. The spectrum of $L$ would be quite different were the boundary conditions not periodic: \emph{a fortiori}, these boundary conditions
therefore play a crucial role in determining the pseudospectra, despite the fact that resolvents corresponding to different boundary conditions
only differ by finite rank terms.

For $f(x)=2\varepsilon x$, the regular-at-the-origin solution  is given explicitly in terms of Bessel functions. This yields a direct construction of our pseudo-modes. For general $f$ we use special transformators \cite{Kravchenko, Volk} to show that the underlying linearised
$f$ still determines the dominant behaviour of our pseudo-modes for large $E$. As always for transformators, it is the $E$-independence
of the underlying kernel which makes them so well adapted to asymptotic analysis.

\section{Summary of main results and scope of the work \label{section:2}}

Throughout this paper, the function $f:[-1,1]\longrightarrow \mathbb{R}$ is twice continuously differentiable with $f(-x)=-f(x)$ and $\sgn(f(x))=\sgn(x)$. Additionally, 
\[
      f(x)= \frac{2\varepsilon x}{1 + x r(x)} 
\] 
where $\varepsilon\in(0,1)$ is a fixed parameter and $r$ is a fixed, twice continuously differentiable function, which is odd and analytic at $x=0$. Note that \[ \frac{1}{f(x)}=\frac{1}{2\varepsilon x}+\frac{r(x)}{2\varepsilon}\] and $f(x)=2\varepsilon x+O(x^3)$ as $x\to 0$; also $f(\pm 1)\not=0$.

We define a differential operator $L$ on a domain in $\operatorname{L}^2(-1,1)$ by 
\[ \mbox{Dom}(L) = \left\{u\in \operatorname{C}([-1,1])\,\,:\,\,\begin{aligned} & fu'+u\in \operatorname{AC}(-1,1)\\ &u(-1)=u(1) \end{aligned}\right\},
\]
with 
\[
    (L u)(x) = (l u)(x) = (f(x)u'(x)+u(x))'\] 
for $u\in \mbox{Dom}(L)$. Note in particular that any $u\in\mbox{Dom}(L)$ must be continuous at the origin, and $2$-periodic
in the sense that $u(-1)=u(1)$.

The following lemma is proved at the end of Section~\ref{sec2}; although its proof involves classical arguments, the result is not an immediate
consequence of existing theory, hence we give full details of its validity.

\begin{lemma} \label{Lemma1}
The densely defined linear operator $L:\operatorname{Dom}(L)\longrightarrow \operatorname{L}^2(-1,1)$ is closed and has compact resolvent.
\end{lemma}

It is convenient to introduce two symmetries commuting with $L$. Let
\[
\mathcal{P}u(x)=u(-x)\qquad \text{and} \qquad\mathcal{T}u(x)=\overline{u(x)},
\]
be the parity and transposition isometries (note that $\mathcal{T}$ is conjugate linear). Then $
\mathcal{P}^2u(x)=\mathcal{T}^2u(x)=u(x)$, and $\mathcal{P}$ and $\mathcal{T}$ leave invariant the subspace $\operatorname{Dom}(L)$. Moreover
\[
    L\mathcal{P}u=-\mathcal{P}Lu \qquad \text{and} \qquad
     L\mathcal{T}u=\mathcal{T}Lu,
\] 
for all $u\in \operatorname{Dom}(L)$. 

Let $\phi\in \operatorname{AC}(-1,0)\cap \operatorname{AC}(0,1)$ and $E\in \mathbb{C}$ be such that
\begin{equation} \label{eq:1}
      (f\phi'+\phi)'=E \phi.
\end{equation}
Then, $(f(\tilde{\phi})'+\tilde{\phi})'=-E\tilde{\phi}$ for $\tilde{\phi}=\mathcal{P}\phi$ and $(f\overline{\phi}'+\overline{\phi})'=\overline{E}\, \overline{\phi}$ for $\overline{\phi}=\mathcal{T}\phi$. Hence, if $E$ is an eigenvalue of $L$ (which additionally requires that $\phi\in\operatorname{Dom}(L)\setminus\{0\}$), then also $-E$ and $\pm \overline{E}$ will be eigenvalues of $L$. Thus, the spectrum of $L$ is symmetric under reflection with respect to the real and imaginary axes. Note that $0$ is always an eigenvalue with corresponding eigenfunction the constant function. Moreover, we have the following analogue of similar results reported in \cite{BLM2010} and \cite{BMR2012} for the case when $f(\pm 1)=0$.

\begin{theorem} \label{lem:specgen}
The spectrum of $L$ is purely discrete and purely imaginary.
\end{theorem} 

The first statement is immediately implied by Lemma~\ref{Lemma1}. The proof of the second statement will be given in Section~\ref{section:3}.

The pseudospectrum of $L$ is also symmetric with respect to the axes. 
This follows from the fact that
\[
       \mathcal{P}(L-E)^{-1}\mathcal{P}=-(L+E)^{-1} \qquad \text{and} \qquad \mathcal{T}(L-E)^{-1}\mathcal{T}=(L-\overline{E})^{-1}.
\]
Indeed, these identities imply that
\[
     \|(L-E)^{-1}\|=\|(L+E)^{-1}\|=\|(L\mp \overline{E})^{-1}\|
\]
for all $E$ in the resolvent set of $L$. 

The upper bound on the resolvent norm given in the next theorem, which we will prove at the end of the paper in Section~\ref{sec6}, is a consequence of general Carleman-type bounds and the fact that the resolvent of $L$ is in the Schatten classes $\mathcal{C}_{p}$ for all $p>\frac23$.

\begin{theorem} \label{th:upperboundpseudospectrum}
Let $\alpha\not\in \Big\{\frac{(2k+1)\pi}{2}\Big\}_{k\in \mathbb{Z}}$ be fixed. For all $p>\frac23$, there exist constants $a,c>0$ such that
\[
     \|(L-|E|\mathrm{e}^{i\alpha})^{-1}\|< c\exp\left[a|E|^{p}\right]
\] 
for all $|E|\geq 1$. The constants $a$ and $c$ can be chosen uniformly for $\alpha$ on a compact set. Both depend on $f(\cdot)$.
\end{theorem}

The main result of this paper is the next theorem, which confirms that the resolvent norm is indeed exponentially large away from the spectrum. It also gives concrete evidence about the shape of the pseudospectrum of $L$. 

\begin{theorem} \label{th:thepseudospectrum}
Let $\alpha\not\in \{\frac{k\pi}{2}\}_{k\in \mathbb{Z}}$ be fixed. There exist constants $a,c>0$ such that
\[
    \|(L-|E|\mathrm{e}^{i\alpha})^{-1}\|>c\exp\left[a |E|^{\frac12} \left| \sin \frac{\alpha}{2}\right|\right] 
\]
for all $|E|\geq 1$. The constant $a$ is independent of $\alpha$ and the constant $c$ can be chosen uniformly for $\alpha$ on a compact set. Both depend on $f(\cdot)$.
\end{theorem}

By making $|E|^{\frac12} \sin \frac{\alpha}{2}$ equal to a constant, then solving for $\alpha$ in terms of $|E|$, we observe that if $E=|E|\mathrm{e}^{i\alpha}$, then
\[|\operatorname{Re}(E)|\sim k_1 |E| \qquad \text{and} \qquad |\operatorname{Im}(E)|\sim k_2 |E|^{\frac12}\]
as $|E|\to \infty$ for the exponential term to become constant on the right hand side. This provides further evidence that the pseudospectra of the operator $L$ include regions that, in the regime $|E|\to\infty$, have boundaries asymptotic to parabolas with directrices the imaginary line and axes of symmetry the real line. This is consistent with the graphs reported in the Figure~6 of \cite{BOS2003} for $f(x)=\frac{2\varepsilon}{\pi}\sin(x)$ and those included in pages 124-125 and 406-408 of \cite{TE2005} for a perturbation of the operator. It also confirms the phenomenon observed in those graphs
that the spacing of the pseudospectral boundaries for level lines $10^{-n}$ is close to linear in $n$.

Davies and Kuijlaars were the first to observe exponential growth of the coefficients in spectral projections for Schr{\"o}dinger operators with complex potential on $\mathbb{R}$, for the complex harmonic oscillator \cite{DaviesHarmOsc}. This work was extended and refined by Henry \cite{Henry1,Henry2}.
Davies \cite{Davies-semiclassical} later used semi-group methods to offer a semi-classical analysis of resolvent growth for families of operators at fixed spectral parameter which, for the particular case of the complex harmonic oscillator, may be combined with a dilation trick to recover exponential resolvent growth results for a fixed operator at large energies. In fact the semi-classical study of pseudospectra at fixed energy had already been taken up independently by Dencker, Sj\"{o}strand and Zworski \cite{DSZ} and became a topic of active research with contributions by Dencker \cite{Dencker},
Helffer \cite{Helffer}, Pravda-Starov \cite{PS}, Hitrik, Sj\"{o}strand and Viola \cite{HSV}, Galkowski \cite{Galk} and Almog and Henry \cite{AH},
among others. More recently, for problems in which there is no semi-classical parameter but the energy goes to infinity, Krej\v{c}i\v{r}\'{\i}k, Siegl, Tater and Viola \cite{MR3416732}, and  Arnal and Siegl \cite{AS} have shown that the exponential growth of the resolvent norm is sharp and prevalent. Their methods do not require smooth potentials, but rather a sufficiently rapid growth of the potential at infinity. 
Mityagin, Siegl and Viola have also returned to the spectral projection results in  \cite{DaviesHarmOsc} and generalised them to a wider class of 1D operators \cite{MR3614165}.

For periodic boundary conditions, the numerical evidence presented by Trefethen in \cite{Trefethen2005} indicates that similar resolvent-norm-growth behaviour also holds for classes of semi-classical Schr{\"o}dinger operators on finite intervals. The article also gives some insights into the shape of the pseudo-modes which might be used to try to prove such a result rigorously.

Our results are not semi-classical. The nature of our problem is such that we are not aware of any simple, rigorous
argument that might be used to obtain large-energy results from fixed-energy semi-classical results for a suitable
family of operators. The proof of Theorem~\ref{th:thepseudospectrum} below involves constructing pseudo-modes which look similar to
those in  \cite{Trefethen2005} away from $x=0$. However, they have the property that there is a jump in the derivatives of order three near the singularity at $x=0$. In Section~\ref{section:4} we present this construction and establish Theorem \ref{th:thepseudospectrum}, for leading order coefficient $f(x)=2\varepsilon x$. In the Appendix~\ref{ApB} we include graphs illustrating the different shapes of these pseudo-modes in terms of $E$.

For general $f$, we first transform the differential equation into a Schr\"{o}dinger equation (Section~\ref{section:4}) 
with singularity of Calogero-type at the origin and a bounded potential. Unfortunately the role of this transformed equation in the 
analysis is such that we cannot use the pseudo-mode constructions of Krej\v{c}i\v{r}\'{\i}k and Siegl \cite{KS2017}.
The singular part is already present for the special case $f(x)=2\varepsilon x$;
the bounded potential comes from the difference $f(x)-2\varepsilon x$. The transformation of the differential equation involves a number of
intermediate steps, so to aid the reader we have included a summary of the variables used and relations between them in Table~\ref{table1} at the end of the paper. 
Having performed these transformations, we use special transformator kernels from \cite{Kravchenko} in Section~\ref{section:3}
to quantify rather precisely the effect of the presence of the bounded potential upon the solutions, and complete the proof of
Theorem~\ref{th:thepseudospectrum} in Section~\ref{section:5}.

\section{The closed operator $L$ \label{sec2}}

In the first part of this section, we show that the solutions to \eqref{eq:1} for any fixed $E\in\mathbb{C}$ are characterised via a change of variables by an ODE whose asymptotic behaviour is driven by Bessel's equation. From this, we will eventually identify a solution which is
bounded. In the forthcoming sections, this bounded solution will determine eigenfunctions and pseudo-modes associated with $L$.

In order to describe the solutions $\phi$ for $E\in \mathbb{C}$, without loss of generality we restrict the variable to $x>0$, then invoke the symmetry $\mathcal{P}$ to determine $\phi(x)$ for $x<0$ from the solution for $-E$. So, set
$
   \phi(x)=F(g(x))
$ 
for $x\in [0,1]$ where $g(\cdot)$, independent of $E$ and $\varepsilon$, is determined up to a scaling by 
\begin{equation}\label{master_eq_g}f(x)g'(x)=\varepsilon g(x). \end{equation}
Then, with $b_1=g(1)$, it follows that $F:[0,b_1]\longrightarrow \mathbb{R}$ must solve
\begin{equation} \label{superBessel}
    (\varepsilon y F'(y))'+F'(y)=E h(y)F(y) \qquad \text{where} \qquad
h(y)=\frac{f\left(g^{-1}(y)\right)}{\varepsilon y}
\end{equation}
is independent of $E$ and $\varepsilon$. 
Indeed, observe that 
 $f(x)\phi'(x)=\varepsilon g(x)F'(g(x))$
and that $(f(x)\phi'(x))'=\varepsilon g'(x)[yF'(y)]'_{y=g(x)}$. As
\[
    E\phi(x)=(f(x)\phi'(x))'+\phi'(x)=g'(x)[(\varepsilon yF'(y))'+F'(y)]_{y=g(x)},
\]
re-ordering this and writing everything in terms of the variable $y$ shows that $F(y)$ satisfies \eqref{superBessel}.
Here and elsewhere the variable $y\in [0,b_1]$. The constant $b_1$ 
is fixed by Lemma \ref{lem:2} below and depends on $r(\cdot)$ but not on $E$ or $\varepsilon$: for example, $b_1=1$ for $r(x)=0$. 

The following lemma fixes the scaling of $g$ and gives its asymptotic behaviour near the origin.
\begin{lemma} \label{lem:2}
The equation \eqref{master_eq_g} has a positive, increasing solution $g\in \mathrm{C}^3((0,1])$, such that\footnote{Here and in all places below, the expression $a(w)\sim b(w)$ as $w\to c$ means that $\displaystyle\lim_{w\to c}\frac{a(w)}{b(w)}=1$. The limit is taken in context, depending on whether $w$ lies in a real or complex set.} $g(x)\sim x^{\frac12}$ and
$g'(x) \sim \frac12 x^{-\frac12}$ as $x\to 0$.
\end{lemma}
\begin{proof}
A solution to \eqref{master_eq_g} is $g(x)=\tilde{g}(x)$ such that
\[
       (\log \tilde{g})'(x)=\frac{1}{2x}+\frac{r(x)}{2}.
\]
Hence we may choose
\[
     \log \tilde{g}(x)=\log x^{\frac12}+r_2(x) \qquad \text{where} \qquad r_2(x)=\frac12 \int_1^x r(s)\,\mathrm{d}s
\]
so that $r_2(x)\sim \frac{a_1}{4} x^2 + a_2$ and $r_2(x)$ is an even analytic  function at $x=0$. This, and the fact that \eqref{master_eq_g} is homogeneous, give a solution 
\begin{equation} \label{eq:functiong}
    g(x)=x^{\frac12}\mathrm{e}^{r_2(x)-a_2}.
\end{equation}
The latter is such that, $g(x)\sim x^{\frac12}$ and $g'(x)\sim \frac{1}{2}x^{-\frac12}$ as $x\to 0$. Moreover, $g(x)$ is $\mathrm{C}^3$, increasing and positive. This follows from the equation and the fact that $f(x)$ is positive for $x>0$ and twice continuously differentiable.
\end{proof}

If we pick $g(x)$ as in this lemma to determine the equation \eqref{superBessel}, then $h(y)$ on the right hand side there is well defined, twice continuously differentiable and
\begin{equation} \label{asympt_h}
    h(y)=2y+O(y^3), \qquad y\to 0.
\end{equation}
Indeed, set $x=g^{-1}(g(x))\sim g^{-1}(x^{\frac12})$. From \eqref{eq:functiong} and the fact that $r_2(x)$ is even and analytic near $x=0$, we have
\[
       g(x)=x^{\frac12} (1+a_2x^2+a_4 x^4+\cdots)
\]
in a neighbourhood of $x=0$. Calling $x=z^2$ and $y=z(1+a_2 z^4+a_4z^8+\cdots)$, gives $z^2=y^2+O(y^4)$. Therefore,
\[
       g^{-1}(y)=y^2+\rho_1(y) \qquad \text{where} \qquad \rho_1(y)=O(y^4).
\]
Thus, the asymptotic behaviour of $f(x)$ and a substitution into the expression for $h(y)$ gives \eqref{asympt_h}.  

As we shall see later, from \eqref{asympt_h} it follows that the behaviour of $\phi(x)$ near $x=0$ will be driven by a Bessel function. This should be expected by re-writing \eqref{superBessel} as
\begin{equation} \label{superBessel2}
     \varepsilon (yF'(y))'+F'(y)=2E (y+\rho_3(y))F(y)
\end{equation} 
where $\rho_3(y)=O(y^3)$ as $y\to 0$.  

Without further mention, everywhere below we will denote by $g:[-1,1]\longrightarrow \mathbb{R}$ the following function. For $x\geq 0$, $g(x)$ will be the solution to \eqref{master_eq_g} with the behaviour near $x=0$ as in Lemma~\ref{lem:2}.
For $x<0$, $g(-x)=g(x)$. In this convention, we then have
\[
    g\in \mathrm{C}([-1,1]) \cap \mathrm{C}^3\big([-1,0)\cup (0,1]\big).
\]

In the remaining part of this section we give the proof of Lemma~\ref{Lemma1}. It will follow from the next statement characterising $\operatorname{Dom}(L)$.

\begin{lemma} \label{lem:Green}
The function $u$ lies in $\operatorname{Dom}(L)$ and satisfies
\begin{equation} \label{eq:bvp}
(fu'+u)'=v
\end{equation} 
for $v\in \operatorname{L}^2(-1,1)$ if and only if 
\begin{equation} \label{eq:orthogonality}
     \int_{-1}^1 \left(g^{\frac{1}{\varepsilon}}(1)-g^{\frac{1}{\varepsilon}}(z)\right)v(z)\mathrm{d}z=0
\end{equation}
and  
\begin{equation} \label{eq:Green}
     u(x)=k+\int_0^x \left( 1-\frac{g^{\frac{1}{\varepsilon}}(z)}{g^{\frac{1}{\varepsilon}}(x)}\right)v(z)\mathrm{d}z 
\end{equation}
for some $k\in \mathbb{C}$.
\end{lemma}
\begin{proof}
We begin the proof by describing the solutions to the homogeneous equation $(f\phi'+\phi)'=0$.
Putting $E=0$ in \eqref{superBessel}, then substituting back $\phi(x)=F(g(x))$, gives a general solution of the form 
\[
     \phi(x)=Ag^{-\frac{1}{\varepsilon}}(x)+B
\]
for all $x\in (0,1)$. Now $\mathcal{P}\phi(x)=\phi(-x)$ must also solve this homogeneous equation. 
Thus a full solution such that $\phi,\,f\phi'+\phi \in \operatorname{AC}(-1,0)\cup \operatorname{AC}(0,1)$ is
\[
     \phi(x)=\begin{cases} A^+ g^{-\frac{1}{\varepsilon}}(x)+B^+, & x\in (0,1), \\
A^{-} g^{-\frac{1}{\varepsilon}}(x)+B^-, & x\in (-1,0).\end{cases}
\]
Now, by variation of parameters, we have that for any $v\in L^1_{\operatorname{loc}}((-1,0)\cup(0,1))$, the solution to \eqref{eq:bvp}
will have to be
\begin{equation} \label{eq:varipar}
     u(x)=-g^{-\frac{1}{\varepsilon}}(x) \left(\int_{\pm 1}^x v(z) g(z)^{\frac{1}{\varepsilon}}\,\mathrm{d}z +k_2^{\pm}   \right)+\left(\int_{\pm 1}^x v(z)\,\mathrm{d}z +k_1^{\pm}\right)
\end{equation}
where the sign is chosen as `$+$' for $x\in (0,1)$ and as `$-$' for $x\in(-1,0)$. With this formula at hand, consider the claim of the lemma.

Let us show the `only if' direction first. Let $v\in \operatorname{L}^2(-1,1)$ and $u\in\operatorname{Dom}(L)$ be related by the identity \eqref{eq:bvp}. By virtue of Lemma~\ref{lem:2}, and since $\frac{1}{\varepsilon}>1$, we know that
\[
 g^{-\frac{1}{\varepsilon}}\not \in \operatorname{L}^2(0,\delta)
\]
for any $\delta>0$. Since $\operatorname{Dom}(L)\subset \operatorname{L}^2(-1,1)$, then necessarily 
\[
    \int_{\pm 1}^0 v(z) g(\pm z)^{\frac{1}{\varepsilon}}\,\mathrm{d}z +k^{\pm}_2=0.
\]
Hence,
\begin{equation} \label{eq:star}
     u(x)=-g^{-\frac{1}{\varepsilon}}(x)\int_0^x v(z)g^{\frac{1}{\varepsilon}}(z)\,\mathrm{d}z+ \int_0^x v(z)\,\mathrm{d}z+k^{\pm}.
\end{equation}
Now, the integrals on the right hand side of this expression are absolutely continuous functions of $x\in[-1,1]$. Moreover, 
\begin{equation} \label{eq:contat0}
\begin{aligned}
   \left|g^{-\frac{1}{\varepsilon}}(x)\int_0^x v(z) g(z)^{\frac{1}{\varepsilon}}\,\mathrm{d}z\right| &\leq |g(x)|^{-\frac{1}{\varepsilon}}\,\|v\| \left(\int_0^x|g(z)|^{\frac{2}{\varepsilon}}\,\mathrm{d}z    \right)^{\frac12} \leq c_3\|v\|\,|x|^{\frac12}
\end{aligned}
\end{equation}
as $x\to 0$. Therefore, since $u$ is continuous, $k^+=k^-$. This yields
\eqref{eq:Green}. Finally, since $u(-1)=u(1)$ and $g$ is an even function,
it follows that \eqref{eq:orthogonality} should hold true.

Now let us show the `if' part. Assume that $v\in \operatorname{L}^2(-1,1)$ satisfies \eqref{eq:orthogonality} and that $u(x)$ is given by \eqref{eq:Green}. As the integral is absolutely continuous for all $x\in [-1,1]$ and $g$ is continuous and non-vanishing for $x\not =0$, then $u$ is a function absolutely continuous on any closed sub-interval of $[-1,0)\cup (0,1]$. But, since \eqref{eq:contat0} holds true, $u$ is also continuous at $x=0$ with $u(0)=k$. That is, $u\in \mathrm{C}([-1,1])$.

Having shown continuity, the identity  \eqref{eq:orthogonality} implies that $u(-1)=u(1)$. Moreover, the construction of the integral representation \eqref{eq:varipar}, of which \eqref{eq:Green} is a particular case, was such that \eqref{eq:bvp} holds true in the distributional sense. But $(fu'+u)'=v\in \operatorname{L}^2(-1,1)$, so indeed $fu'+u\in \operatorname{AC}(-1,1)$. This yields $u\in \operatorname{Dom}(L)$ and completes the proof of Lemma \ref{lem:Green}.
\end{proof}

In this lemma, the conditions \eqref{eq:orthogonality} and \eqref{eq:Green} are compatible with those given in \cite[Proposition~2.2]{ChugKarabas} or the more general \cite[Lemma~4]{BMR2012}. Indeed, whenever $f(\pm 1)=0$, which is not covered in the present paper, periodicity of the functions in the domain would require\footnote{Here and everywhere below, $\langle v\rangle=\frac12 \int_{-1}^1v(x)\,\mathrm{d}x$.}  $\langle v\rangle=0$ instead of \eqref{eq:orthogonality} due to the fact that $g(x)$ would also vanish at $x=\pm 1$.  Note that we only require $u(-1)=u(1)$ in the domain, irrrespectively of whether $f(\pm1)$ vanishes, since the endpoints $\pm 1$ are of limit-circle type for $f(\pm 1)=0$.

We now show that $\operatorname{Dom}(L)$ is indeed a domain of closure for $L$ and that the resolvent of $L$ is compact. 

\begin{proof}[Proof of Lemma~\ref{Lemma1}] According to Lemma~\ref{lem:Green}, we know that
\[
    \operatorname{Ker}(L)=\operatorname{Span}\{1\}\subset \operatorname{L}^2(-1,1).
\]
We show that the reduced operator
\[
    \tilde{L}: \operatorname{Dom}(\tilde{L})=\operatorname{Dom}(L)\cap \{1\}^\perp \longrightarrow \operatorname{Ran}(L)
\]
is invertible and its inverse is compact.

{}Firstly, note that $v\in \operatorname{Ran}(L)$ if and only if \eqref{eq:orthogonality} holds true. Then 
\[
   \operatorname{Ran}(L)=\{g^{\frac{1}{\varepsilon}}(1)-g^{\frac{1}{\varepsilon}}(x)\}^{\perp}.
\]
Therefore, $\operatorname{Ran}(L)$ is a closed subspace of $\mathrm{L}^2(-1,1)$ with codimension 1. Let $S:\operatorname{Ran}(L)\longrightarrow \operatorname{Dom}(L)$ be given by
\[
     Sv(x)=\int_{-1}^1 H(x,z) v(z)\mathrm{d}z
\]
where 
\begin{equation} \label{kernelH}
     H(x,z)=\operatorname{sgn}(x)\left( 1-\frac{g^{\frac{1}{\varepsilon}}(z)}{g^{\frac{1}{\varepsilon}}(x)}\right) \mathds{1}_{(0,|x|]}(\operatorname{sgn}(x) z).
\end{equation}
That is, $Sv(x)=u(x)$ in formula \eqref{eq:Green} with the constant $k=0$.
By virtue of Lemma~\ref{lem:2}, we know that
\[
     \sup_{-1\leq x,z\leq 1} |H(x,z)| <\infty. 
\]
Thus, $S$ is a compact operator. 

Let 
\[
     \tilde{S}v(x)=Sv(x)-\langle Sv\rangle.
\]
Then, $\tilde{S}:\mathrm{Ran}(L)\longrightarrow \operatorname{Dom}(\tilde{L})$. Indeed, by virtue of \eqref{eq:Green}, for all $v$ satisfying \eqref{eq:orthogonality} we have $\tilde{S}v=u\in \operatorname{Dom}(L)$ and $\langle \tilde{S}v\rangle=0$.  Moreover, $\tilde{S}$ is the inverse of $\tilde{L}$. Indeed, according to Lemma~\ref{lem:Green}, $\tilde{L}\tilde{S}v=L\tilde{S}v=v$ for all $v$ satisfying \eqref{eq:orthogonality} and $\tilde{S}\tilde{L}=\tilde{S}(Lu)=u$ for all $u\in\operatorname{Dom}(\tilde{L})\subset \operatorname{Dom}(L)$. This confirms that $\tilde{L}$ is invertible. Finally, as $\tilde{S}$ is a rank one perturbation of $S$, then it is also a compact operator.  \end{proof}

\section{Proof of theorems~\ref{lem:specgen} and \ref{th:thepseudospectrum} for leading order \label{section:4}}

In this section we describe the spectrum and pseudospectrum of $L$ in the specific case $r(x)=0$; that is, $f(x)=2\varepsilon x$. We begin by finding a regular-at-the-origin solution to \eqref{eq:1}, denoted $\Phi(x;E)$ or just $\Phi(x)$,
such that $\Phi(0)=1$. 

Since $r(x)=0$ yields $g(x)=|x|^{\frac12}$ and $h(y)=2y$ in \eqref{superBessel}, the latter reduces to Bessel's equation
\[
     \varepsilon (yF'(y))'+ F'(y)= 2E yF(y).
\]
Let 
\begin{equation} \label{eq:change}
      G(y)=y^{\beta}F(y) \qquad\text{for} \qquad  \beta=\frac{1}{2\varepsilon}+\frac12.
\end{equation}
Since
\begin{align*}
G'(y)&=\beta y^{\beta-1}F(y)+y^{\beta}F'(y) \qquad \text{and} \\
G''(y)&=y^{\beta-1}\Big((yF'(y))'+(2\beta-1)F'(y)+\frac{\beta(\beta-1)}{y}F(y)\Big),
\end{align*}
it follows that
\begin{equation} \label{superBessel0}
G''(y)+\left(\lambda^2-\frac{m^2-\frac14}{y^2}\right)G(y)=0 
\end{equation}
where, from now on, we set the parameters
\begin{equation} \label{eq:ren}
m=\frac{1}{2\varepsilon} \qquad \text{and} \qquad \lambda^2=-4mE.
\end{equation} 

An unscaled solution to \eqref{superBessel0} is \cite[10.13.1]{NIST}
\[
      \widetilde{G}(y)=y^{\frac12}J_m(\lambda y),
\]
giving $\widetilde{F}(y)=y^{-m}J_m(\lambda y)$.
Since near $y=0$
\[
     \widetilde{F}(y)=\frac{\lambda^m J_m(\lambda y)}{(\lambda y)^m}\sim
\frac{\lambda^m}{2^m \Gamma(m+1)}=\widetilde{F}(0),
\]
we obtain that
\[
   F(y) = \frac{2^m \Gamma(m+1)}{\lambda^{m}} y^{-m}J_m(\lambda y)
\]
is such that $F(0)=1$. Since $\Phi$ should be analytic at $x=0$,
then irrespective of whether $x>0$ or $x<0$, 
\begin{equation} \label{linearPhi}
     \Phi(x)=\frac{2^{m}\Gamma(m+1)}{\lambda^{m}}x^{-\frac{m}{2}}J_{m}(\lambda x^{\frac12}),  \qquad \qquad x\in[-1,1],
\end{equation}
is a solution to \eqref{eq:1} for $r(x)=0$ such that $\Phi(0)=1$. Expanding the Bessel function in a power series gives
\begin{equation}  \label{TaylorlinearPhi}
     \Phi(x;E)=\Phi(x)=\sum_{k=0}^\infty (-1)^k \frac{\Gamma(m+1)}{k!\Gamma(m+k+1)}\left(\frac{\lambda}{2}\right)^{2k} x^k
\end{equation}
which converges for all $x\in[-1,1]$ (in fact, for all $x\in \mathbb C$). This is the regular solution required.

\begin{figure}[t]
\includegraphics[width=60mm]{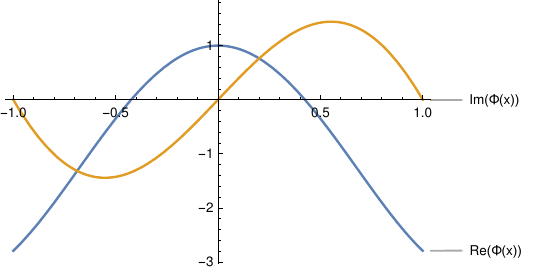} \hspace{.5cm}
\includegraphics[width=35mm]{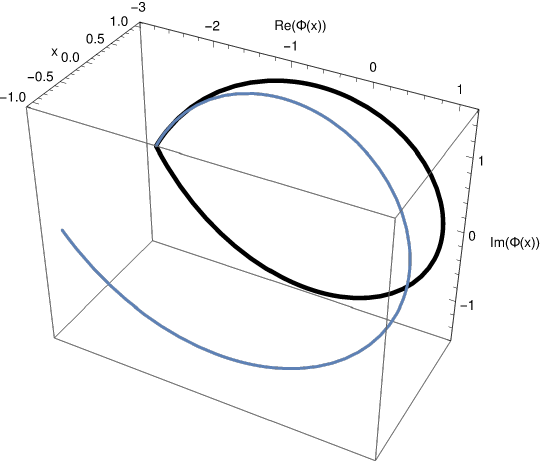}
\caption{For $f(x)=\frac{x}{2}$, the graphs show the eigenfunction $\Phi(x)$ associated with the first non-zero eigenvalue $E\approx 47.8853i$.}
\end{figure}

The proof of the next particular case of Theorem~\ref{lem:specgen} follows the same argument as the proof of \cite[Theorem~2.1]{BLM2010}, but in this case $f(\pm1)\not=0$. 

\begin{lemma} \label{lem:speclin}
Let $f(x)=2\varepsilon x$ for $0<\varepsilon<1$. Then, the spectrum of $L$ is purely imaginary. 
\end{lemma}
\begin{proof}
The complex number $E$ is an eigenvalue of $L$ if and only if $\Phi(-1;E)=\Phi(1;E)$. This is equivalent to
\[
     \mathrm{e}^{\frac{im\pi}{2}} J_{m}(i\lambda)=J_m(\lambda).
\]
Hence, a necessary condition for $E$ to be an eigenvalue is that 
\begin{equation} \label{eq:condeva}
    \left|\frac{J_{m}(i\lambda)}{J_{m}(\lambda)}\right|=1.
\end{equation}
Now, we know that 
$
J_m(\overline{z})=\overline{J_m(z)}$ and $J_m(z\mathrm{e}^{i\pi})=\mathrm{e}^{i\pi m}J_m(z)
$ for all $z\in \mathbb{C}$. Then, \eqref{eq:condeva} holds true for all $\lambda$ such that  $\arg(\lambda)\in\{\pm \frac{\pi}{4},\,\pm \frac{3\pi}{4}\}$. The proof of the lemma will follow from the fact that the latter is also necessary for the identity \eqref{eq:condeva}.

Indeed, according to \cite[10.9.4]{NIST},
\[ J_m(z) = \frac{2\left(\frac{1}{2}z\right)^m}{\sqrt{\pi}\Gamma\left(m+\frac{1}{2}\right)}
\int_0^1 (1-t^2)^{m-\frac12}\cos(zt)dt. \]
Recall that $m>\frac{1}{2}$, $|\cos(zt)|\leq \exp(|z|)$ for $0\leq t \leq 1$ and $\int_0^1(1-t^2)^{m-\frac12}dt \leq 1$. Then
\begin{equation} \label{lem:Besselorder1}
 |J_m(z)| \leq \frac{2^{1-m}|z|^m \exp(|z|)}{\sqrt{\pi}\Gamma\left(m+\frac{1}{2}\right)}. 
 \end{equation}
 Hence, the analytic function $J_m(z)$ has growth order $1$. The function
\[
      \lambda \longmapsto \frac{\mathrm{e}^{\frac{im\pi}{2}} J_{m}(i\lambda)}{J_m(\lambda)}
\]
is therefore a meromorphic function with poles on the real axis (the Bessel zeros), zeros on the imaginary axis and it has growth order less than or equal to 1 in suitable sectors of the plane. By the Phragm{\'e}n-Lindel{\"o}f Principle and the fact that (\ref{eq:condeva}) holds for
all $\lambda$ with  $\arg(\lambda)\in\{\pm \frac{\pi}{4},\,\pm \frac{3\pi}{4}\}$, 
\[
     \left|\frac{J_{m}(i\lambda)}{J_{m}(\lambda)}\right|<1 \qquad \text{for}\qquad \frac{\pi}{4}<\arg(\pm\lambda)<\frac{3\pi}{4}.
\]
Moreover, replacing $\lambda$ by $i\lambda$, necessarily the opposite inequality
holds,
    \[
     \left|\frac{J_{m}(i\lambda)}{J_{m}(\lambda)}\right|>1 \qquad \text{for}\qquad -
\frac{\pi}{4}<\arg(\pm\lambda)< \frac{\pi}{4}.\]
Hence, \eqref{eq:condeva} is only possible when $\arg(\lambda)\in\{\pm \frac{\pi}{4},\,\pm \frac{3\pi}{4}\}$. This completes the proof of the lemma.  
\end{proof}

We now develop asymptotic resolvent norm estimates for $L$. These describe the  pseudospectrum of $L$ away from the spectrum. Since $L$ and the symmetries $\mathcal{P}$ and $\mathcal{T}$ commute,  without loss of generality we can assume that $E$ is in the open first quadrant. Recalling \eqref{eq:ren}, we therefore set  
\begin{equation} \label{lambdaandE}
        \lambda=|\lambda|\mathrm{e}^{i\left(\frac{\pi}{2}+\theta\right)} \qquad \text{for} \quad \theta\in\Big(0,\frac{\pi}{4}\Big). 
\end{equation}
Then,
\[
   \frac{\Phi(-1)}{\Phi(1)}=\mathrm{e}^{-\frac{i\pi m}{2}}\frac{J_m(-|\lambda|\mathrm{e}^{i\theta})}{J_m(i|\lambda|\mathrm{e}^{i\theta})}.
\]
Now, we recall \cite[10.7.8]{NIST} that
\begin{equation} \label{preasympBessel}
     J_{m}(z)= \frac{\mathrm{e}^{i(z-\frac{\pi}{4}-m\frac{\pi}{2})}+\mathrm{e}^{-i(z-\frac{\pi}{4}-m\frac{\pi}{2})}}{\sqrt{2\pi}z^{\frac12}}+\frac{\sqrt{2}\mathrm{e}^{|\operatorname{Im}z|}}{\sqrt{\pi}z^{\frac12}}o(1)
\end{equation}
as $z\to \infty$ in sectors $|\operatorname{arg}(z)|\leq \pi-\delta$, where the limit is uniform for any fixed $0<\delta<\pi$.
Thus, for all $\tau>0$ and $\theta\in \big(0,\frac{\pi}{4}\big)$ fixed, 
\begin{equation} \label{asympBessel}
      \left|J_m(-|\lambda|\tau \mathrm{e}^{i\theta})\right|\sim \frac{\mathrm{e}^{|\lambda|\tau \sin \theta}}{\sqrt{2\pi}|\lambda|^{\frac12}\tau^{\frac12}} \qquad \text{and} \qquad \left|J_m(i|\lambda|\tau \mathrm{e}^{i\theta})\right|\sim \frac{\mathrm{e}^{|\lambda|\tau \cos \theta}}{\sqrt{2\pi}|\lambda|^{\frac12}\tau^{\frac12}}
\end{equation}
as $|\lambda|\to\infty$. Moreover, these limits are uniform in $\theta$; for all $\theta\in\big[\delta,\frac{\pi}{4}\big]$ on the left hand side and for all $\theta\in\big[0,\frac{\pi}{4}-\delta\big]$ on the right hand side. Hence, for $\theta\in (0,\frac{\pi}{4})$, uniformly in $\big[\delta,\frac{\pi}{4}-\delta\big]$, we know that
\begin{equation} \label{asympPhi-1Phi1}
      \left|\frac{\Phi(-1)}{\Phi(1)}\right|\sim \mathrm{e}^{|\lambda|(\sin\theta-\cos \theta)} 
\end{equation}
as $|\lambda|\to\infty$. Note that, this estimate is always decaying in $|\lambda|$. We also remark at this stage that we know that $|\frac{\Phi(-1)}{\Phi(1)}|<1$ for all $|\lambda|>0$ and $\theta$ in the sector considered here, from the proof of Lemma~\ref{lem:speclin}.

\begin{figure}[t]
(a) \includegraphics[width=.45\textwidth]{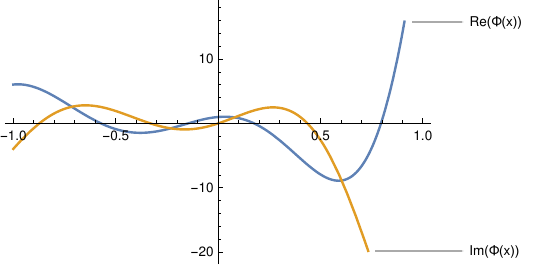} 
\includegraphics[width=.45\textwidth]{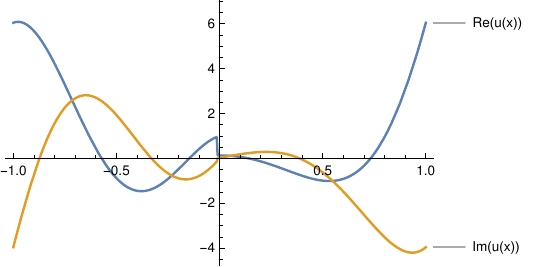}(c)

(b) \includegraphics[width=.45\textwidth]{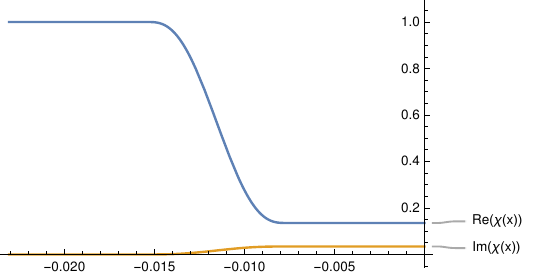} 
\includegraphics[width=.45\textwidth]{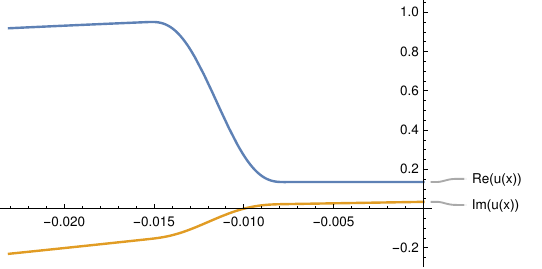}(d)
\caption{Illustration of the construction of the pseudo-mode for $f(x)=\frac{x}{2}$ and $\lambda=-7+9i$. (a)~Solution $\Phi(x)$, (b)~periodiser $\chi(x)$ in the vicinity of $[-\frac{2}{|\lambda|^2},-\frac{1}{|\lambda|^2}]$, pseudo-mode $\quasi(x)$ (c)~in $[-1,1]$ and (d)~in the vicinity of $[-\frac{2}{|\lambda|^2},-\frac{1}{|\lambda|^2}]$.}\label{fig2}
\end{figure}

The following proposition gives Theorem~\ref{th:thepseudospectrum} for linear $f(x)$. The constant $R$ below is independent of $\lambda$, $m$ and $\delta$.

\begin{proposition}\label{claim:2}
Let $f(x)=\frac{x}{m}$ where $m>\frac12$. Let $\lambda\in \mathbb{C}$, lie in the sector prescribed by \eqref{lambdaandE} and set $E\equiv E_\lambda=-\frac{\lambda^2}{4m}$. Let $0<\delta<\frac{\pi}{8}$ be fixed. There exists a constant $R>0$ ensuring the following.  For all $|\lambda|\geq R$ and $\theta\in \big[\delta,\frac{\pi}{4}-\delta\big]$, we can find a pseudo-mode $\quasi\equiv \quasi_{\lambda}\in \operatorname{Dom}(L)$ such that 
\begin{equation}  \label{pseudospecasymp}
    \frac{\|(L-E_\lambda)\quasi_{\lambda}\|}{\|\quasi_{\lambda}\|}\leq
\frac{2\sqrt{2\pi}\mathrm{e}^{\frac12}\big(\frac{32}{m}+4\big)}{2^m\Gamma(m+1)} |\lambda|^{m+\frac32} \mathrm{e}^{-\frac{|\lambda|}{\sqrt{2}}\sin \theta} .
\end{equation} 
\end{proposition} 
\begin{proof} The proof is split into three main steps.  

\underline{Step 1}: construction of $\quasi$. We write $\quasi(x)$ as the product of a periodiser $\chi(x)$ and the bounded solution $\Phi(x)$ as follows, see Figure~\ref{fig2}. 
Let 
\[
 p(x)=\frac{\int_x^1 t^2(1-t)^2\,dt}{\int_0^1 t^2(1-t)^2\,dt} = 1-10x^3+15 x^4-6x^5.
\] Then, $p(0)=1$, $p'(0)=p''(0)=0$, $p(1)=p'(1)=p''(1)=0$ and $p(x)$ is decreasing on $[0,1]$. 
Moreover, 
\[
    \sup_{x\in[0,1]}|p'(x)|=\frac{15}{8}<2 \qquad \text{and} \qquad
\sup_{x\in[0,1]}|p''(x)|<6
\] 
where the latter supremum is achieved at $x=\frac{3\pm \sqrt{3}}{6}$. 
Let
\begin{equation}\label{eq:chidef}
    \chi(x)=\begin{cases} 1, & -1 \leq x \leq -\frac{2}{|\lambda|^2}, \\ \\
     \left[\left(1-\frac{\Phi(-1)}{\Phi(1)}\right) p\big(|\lambda|^2 x+2\big)+\frac{\Phi(-1)}{\Phi(1)}\right], & -\frac{2}{|\lambda|^2}\leq x \leq -\frac{1}{|\lambda|^2}, \\ \\
\frac{\Phi(-1)}{\Phi(1)}, & -\frac{1}{|\lambda|^2} \leq x \leq 1.   \end{cases}
\end{equation}
Then $\chi\in \mathrm{C}^2([-1,1])$. Moreover, from the upper bounds on the derivatives of $p(x)$ above, we get that for all $x\in[-1,1]$, 
\begin{equation} \label{star}
     |\chi'(x)|\leq \frac{15}{4}|\lambda|^2
\qquad \text{and}
\qquad
    |\chi''(x)|\leq 12|\lambda|^4.
\end{equation}
Now, set $\quasi(x)=\chi(x) \Phi(x)$. Note that by construction 
$\quasi\in \mathrm{C}^2([-1,1])$ and $\quasi(-1)=\quasi(1)$ so $\quasi\in \operatorname{Dom}(L)$.

\underline{Step 2}: lower bound on the denominator of \eqref{pseudospecasymp}. We now show that there exists $R>0$ such that
\begin{equation} \label{lowerdenominator}
    \|\quasi\|\geq \frac{2^{m}\Gamma(m+1)}{2\sqrt{2\pi}|\lambda|^{m+\frac12}} \mathrm{e}^{\frac{|\lambda|\sin \theta}{\sqrt{2}}}  
\end{equation}
for all  $|\lambda|\geq R$. Indeed, let $x\in[-1,-\frac12]$. From the expression for $\Phi(x)$ in \eqref{linearPhi} and the restriction on the argument of $\lambda$ from \eqref{lambdaandE}, it follows that
\[
     |\Phi(x)|\geq \frac{2^m \Gamma(m+1)}{|\lambda|^{m}} \left| J_{m}\left( -|\lambda||x|^{\frac12} \mathrm{e}^{i\theta}   \right)   \right|.
\]
According to the first identity in \eqref{asympBessel}, the right hand side of this is
\[
    \sim \frac{2^{m} \Gamma(m+1)}{\sqrt{2\pi}|\lambda|^{m+\frac12}|x|^{\frac14}} \mathrm{e}^{|\lambda||x|^\frac12 \sin \theta}
\]
for $|\lambda|\to\infty$ and fixed $x\in[-1,-\frac12]$. Then, for all $c>1$, there is a sufficiently large $R>0$ such that
\[
     |\Phi(x)|\geq c \frac{2^{m} \Gamma(m+1)}{\sqrt{2\pi}|\lambda|^{m+\frac12}} \mathrm{e}^{\frac{|\lambda| \sin \theta}{\sqrt{2}}}
\]
for all  $x\in\left[-1,-\frac12\right]$ and $|\lambda|> R$. Integrating in mean square and taking square roots, gives \eqref{lowerdenominator}.

\underline{Step 3}: upper bound on the numerator of \eqref{pseudospecasymp}. We now show that 
\begin{equation}   \label{uppernumerator}
    \|(L-E)\quasi\|\leq |\lambda|\mathrm{e}^{\frac12} \left(\frac{32}{m}+4\right)
\end{equation}
for all $|\lambda|\geq 2$. For this, we begin by noting that
\[
     (L-E)\quasi=\chi(L-E)\Phi+\Phi[L\chi]+2f\chi'\Phi'=\Phi[L\chi]+2f\chi'\Phi'.
\]
From the Maclaurin expansion \eqref{TaylorlinearPhi} of $\Phi(x)$, it follows that,
\begin{align*}
   |\Phi(x)|& \leq \sum_{r=0}^\infty \frac{2^r}{2^{2r}r!}=\mathrm{e}^{\frac12}
\end{align*}
for all $x\in\left[-\frac{2}{|\lambda|^2},-\frac{1}{|\lambda|^2}\right]$ whenever $|\lambda|\geq 2$. 
Now
\[
   \Phi'(x)=\sum_{r=0}^\infty (-1)^{r+1}\frac{\Gamma(m+1)}{r!\Gamma(m+r+2)}\left(\frac{\lambda^2}{4}\right)^{r+1} x^r.
\]
Then,
\begin{align*}
    |\Phi'(x)|&\leq \sum_{r=0}^\infty \left(\frac{|\lambda|^2}{4}\right)^{r+1}\frac{1}{r!}\left(\frac{2}{|\lambda|^2}\right)^r=\frac{|\lambda|^2}{4}\mathrm{e}^{\frac12}
\end{align*}
for all such $x$ and $\lambda$ too. 
Next, using $2\varepsilon |x| \leq 2 \varepsilon \frac{2}{|\lambda|^2}$, along with \eqref{star},
we obtain 
\[  |L[\chi](x)| = |2\varepsilon x\chi''(x)+(2\varepsilon+1)\chi'(x)|\leq 2\varepsilon \frac{2}{|\lambda|^2}12|\lambda|^4 + (2\varepsilon+1)
\frac{15}{4}|\lambda|^2 \]
\[
 = \left(55\frac{1}{2}\varepsilon+\frac{15}{4}\right)|\lambda|^2 = \left(\frac{27\frac{3}{4}}{m}+\frac{15}{4}\right)|\lambda|^2. \]
Therefore, integrating in the segment where $\chi(x)$ is not constant gives the following:
\begin{equation}\label{eq:neededlater}
\begin{aligned}
\|(L-E)\quasi\|&\leq \|\Phi L[\chi]\|+\|2f\chi' \Phi' \| \\
     &\leq \frac{1}{|\lambda|}\left(\|\Phi L[\chi]\|_{\infty}+\frac{4}{m|\lambda|^2}    \|\chi'\Phi'\|_{\infty}\right) \\
     &\leq \mathrm{e}^{\frac12}\left(\frac{27\frac{3}{4}}{m}+\frac{15}{4}\right)|\lambda| + \frac{\frac{15}{4}|\lambda|}{m}\mathrm{e}^{\frac12} 
      \; \leq \; |\lambda|\mathrm{e}^{\frac12}\left(\frac{32}{m}+4\right),
\end{aligned} \end{equation}
for $|\lambda|\geq 2$. 

Combining steps~2 and 3, the statement in the proposition is confirmed. Note that the claim that $R$ can be chosen uniform in $\theta\in[\delta,\frac{\pi}{4}-\delta]$, follows from the analogous property already confirmed for \eqref{asympBessel}.
\end{proof}

As a consequence of the above, we see that for $E\in\mathbb{C}$ with fixed $\arg(E)\not\in\{\frac{k\pi}{2}\}_{k\in \mathbb{Z}}$, the resolvent norm $\|(L-E)^{-1}\|\to \infty$ in the context of Theorem~\ref{th:thepseudospectrum} with $a=\frac{\varepsilon^{\frac12}}{4}$.

\section{Integral representation of the solution and proof of Theorem~\ref{lem:specgen} \label{section:3}}
For the remainder of this paper we return to general $f$. Applying several changes of variables, whose parameters are summarised in Table~\ref{table1}, we re-write \eqref{superBessel} as a Schr\"odinger equation with a perturbed Calogero potential on a finite segment.
We derive an integral representation of the regular-at-zero solution, by applying the theory of transformators to this Schr\"{o}dinger
equation. We will then see that the growth-order $1$ of $J_m(z)$ (see \eqref{lem:Besselorder1}) carries over to the solutions for general $f$, allowing us to show that the spectrum of $L$ is purely imaginary. The integral representation
of solutions also carries over the finer asymptotic estimates from the case $f(x)=2\varepsilon x$. Therefore, the properties of a pseudo-mode constructed in the same fashion as in Propopsition~\ref{claim:2} are preserved, as we shall see in Section~\ref{section:5}.

For non-linear $f$, the change of variables \eqref{eq:change} leads to
\begin{equation} \label{superBessel3}
G''(y)-\frac{m^2-\frac14}{y^2}G(y)=4mE (1+\rho(y))G(y) 
\end{equation}
where $y\in [0,b_1]$ and $\rho(y)\not\equiv 0$ is given by 
\[\rho(y)=\frac{\rho_3(y)}{y}=\frac{h(y)}{y}-2=O(y^2)\] as $y\to 0$. The function $h$ is given in
 \eqref{superBessel} and satisfies \eqref{asympt_h}, and recall that $b_1=g(1)$.
This intermediate problem can now be re-written in Liouville normal form, as follows.

\begin{lemma} \label{lemma:5}
The change of variables 
\begin{equation} \label{eq:LNF}
     t=\int_0^y (1+\rho(s))^{\frac12}\,\mathrm{d}s \qquad\text{and} \qquad Z(t)=(1+\rho(y))^{\frac14}G(y),
\end{equation}
transforms the equation \eqref{superBessel3} into the equation
 \begin{equation}\label{eq:zLNF3}
 -Z''(t) + q(t) Z(t) + \frac{\ell(\ell+1)}{t^2} Z(t) = -4mE Z(t)
\end{equation}
for $t\in(0,b_2]$ and 
\[
 \ell= \frac{1}{2\varepsilon}-\frac{1}{2} = m - \frac{1}{2} >0, \]
 where $b_2=\int_0^{b_1}(1+\rho(s))^{\frac12} \,\mathrm{d}s$ and
the potential $q:[0,b_2]\longrightarrow \mathbb{R}$ is continuous and bounded.
\end{lemma}
\begin{proof}
Firstly note that 
\begin{eqnarray}
\hspace{-1cm} Z''(t) - \frac{m^2-\frac{1}{4}}{y^2}Z(t)  &=&  4mE Z(t) + \nonumber \\
& & \left[ \frac{\left(\frac{\rho'(y)}{4}(1+\rho(y))^{-\frac{5}{4}}\right)'}{(1+\rho(y))^{\frac{3}{4}}}-\frac{(m^2-\frac{1}{4})\rho(y)}{y^2(1+\rho(y))}\right] Z(t).  \label{eq:zLNF}
\end{eqnarray}
We wish to replace the coefficient $\displaystyle \frac{m^2-\frac{1}{4}}{y^2}$ by 
$\displaystyle \frac{m^2-\frac{1}{4}}{t^2}$ on the left hand side. Write
\begin{equation*}
\hspace{-2pt} q(t) = \left(m^2-\frac{1}{4}\right)\left(\frac{1}{t^2}-\frac{1}{y^2}\right) + 
\left[ \frac{\left(\frac{\rho'(y)}{4}(1+\rho(y))^{-\frac{5}{4}}\right)'}{(1+\rho(y))^{\frac{3}{4}}}-\frac{(m^2-\frac{1}{4})\rho(y)}{y^2(1+\rho(y))}\right].
\end{equation*} 
Since $f$ is twice continuously differentiable, $\rho''$ and hence $q$ are continuous (and thus bounded)
outside any neighbourhood of the origin.

To show that $q$ is bounded for all $t\in[0,b_2]$, therefore, we need only check that it is bounded at $t=0$. 
For $\rho(y)\sim \alpha y^2$, where $y$ is small,
\begin{equation}
\label{eq:tysmall}
t = y + \frac{\alpha}{6}y^3 + \cdots \qquad \text{and} \qquad y = t - \frac{\alpha}{6}t^3 + \cdots  
\end{equation}
where $t$ is small. Then,
\[
\frac{1}{y^2} = \frac{1}{t^2}\left(1-\frac{\alpha}{6}t^2 + \cdots\right)^{-2} = \frac{1}{t^2}\left(1 + \frac{\alpha}{3}t^2 + O(t^4)\right) = 
 \frac{1}{t^2} + \frac{\alpha}{3} + O(t^2). 
 \]
Hence, $\displaystyle \frac{1}{t^2}-\frac{1}{y^2}$ is $O(1)$ for small $t$. 
Moreover, for small $y$, $\frac{\rho(y)}{y^2} = \alpha=O(1)$. Furthermore, all other
terms in the expression for $q(t)$ are $O(1)$ for small $t$. Thus $q$ is indeed bounded.

Finally, since
\[ m^2-\frac{1}{4} = \left(\frac{1}{2\varepsilon}-\frac{1}{2}\right) \left(\frac{1}{2\varepsilon}+\frac{1}{2}\right)  = \ell(\ell+1),
\]
the expression \eqref{eq:zLNF3} is confirmed.
\end{proof}

\begin{remark} \label{rem:domainZeq}
Denote $Z(t)$ and $\Phi(x)$ by $Z(t;E)$ and $\Phi(x;E)$, to indicate explicitly the dependence on $E$. 
In the proof of Theorem~\ref{lem:specgen} below, we shall study the zeros of the map $E\mapsto \Phi(1;E)$. These will be 
expressed in terms of the eigenvalues of  \eqref{eq:zLNF3} on $(0,b_2]$ with Dirichlet boundary condition $Z(b_2;E)=0$. The 
associated operator is self-adjoint when $m\geq 1$ since the singularity at the origin is then of limit-point type \cite{Coddington}. 
For $\frac12<m<1$ the singularity at the origin is of limit-circle-nonoscillatory type; its unique self-adjoint Friedrichs extension has 
eigenfunctions which are the `principal solutions' at the origin \cite{Rellich} and so coincide with the transformation of $\Phi(x;E)$.
\end{remark}

We now proceed to wrap up all the transformations by expressing $\Phi(x)$ in terms of $Z(t)$.
Since $\Phi(x) = F(g(x))$ and $G(y)=y^\beta F(y)$, we have 
\[ \Phi(x) = g(x)^{-\beta}G(g(x)). \]
Since \eqref{eq:LNF} gives
\[ G(y) = \big(1+\rho(y)\big)^{-\frac14}\,\, Z\left(\int_0^y \sqrt{1+\rho(s)}\right), \]
then, as $\beta = \frac{1}{2}+\frac{1}{2\varepsilon}= \ell+1$,
\begin{equation}\label{eq:atg}
 \Phi(x) = (1+\rho(g(x)))^{-\frac14}g(x)^{-(\ell+1)}\,\, Z\left(\int_0^{g(x)} \sqrt{1+\rho(s)}ds\right). 
 \end{equation}
In order to ensure $\Phi(0)=1$, here and everywhere below we choose a solution $Z(t)$ of  \eqref{eq:zLNF3} 
such that 
\begin{equation} \label{asympZat0} 
Z(t)= t^{\ell+1}+o(t^{\ell+\frac52})
\end{equation}
as $t\to 0$ (a `principal solution'). It is straightforward to see that this solution exists. 

Indicating the dependence on $q(\cdot)$ and $E$ explicitly, we use $Z_q(\cdot;E)$ to denote the solution of \eqref{eq:zLNF3}. In this notation,  
\begin{equation} \label{eq:Z0}
Z_0(t;E)=\frac{2^m\Gamma(m+1)}{\lambda^m}t^{\frac12}J_m(\lambda t) \end{equation} 
 is the solution for $q(t)=0$. Recall that $\lambda^2 = -4mE$.
Let
\[\tau(x) = \int_0^{g(x)}\sqrt{1+\rho(s)}\,\mathrm{d}s,\]
so that $\tau(x)\sim|x|^{\frac12}$ as $x\to 0$. It will be convenient to re-write \eqref{eq:atg} as
\begin{equation} \label{eq:atgprime} \tag{\ref{eq:atg}$'$} 
\Phi(x;E) = \rho_4(g(x))Z_q(\tau(x);E) \qquad \text{for} \qquad \rho_4(y)=(1+\rho(y))^{-\frac14}y^{-(\ell+1)}. \end{equation}
Note that $\rho_4(y)=O(y^{-(\ell+1)})$ as $y\to 0$.

Our final task in this section before proving Theorem~\ref{lem:specgen}, is to write $Z_q(t;E)$ in terms of $Z_0(t;E)$. For this, we shall need the `transformator equation'
\begin{equation}\label{eq:transK}
Z_q(t;E) = Z_0(t;E) + \int_0^t K_q(t,s)Z_0(s;E) ds 
\end{equation}
in which $K_q$ is an $E$-independent kernel, described in \cite{Kravchenko} and dating back to \cite{Volk},
to map the known asymptotic results
on $Z_0$ to corresponding results for $Z_q$. Before doing so, we note the following properties of $K_q$.
\begin{enumerate}
\item[i)] If $q$ is continuous including at $0$,
then a solution $Z_q(t;E)$ of \eqref{eq:zLNF3} is given by (\ref{eq:transK}).
\item[ii)] \label{supnormKq} \[\sup_{0\leq s\leq t\leq b_2} |K_q(t,s)|<\infty.\]
\item[iii)] $K_q(t,s)$ solves \[\Big(\partial_t^2-\partial_s^2 - \frac{\ell(\ell+1)}{t^2} + \frac{\ell(\ell+1)}{s^2} - q(t)\Big)K_q=0\]
in the region $0<s < t$ with boundary conditions $K_q(t,0) = 0$ for $t\geq 0$  and $2 \frac{d}{dt}K_q(t,t) = q(t)$.
\end{enumerate}
In the sequel, we shall only use the properties i) and ii). The property i) follows from Lemma~\ref{lemma:5}.

\begin{proof}[Proof of Theorem~\ref{lem:specgen}]
Firstly, recall that that $E\in \operatorname{Spec}(L)$ if and only if
\[
     \Phi(-1;E)=\Phi(1;E).
\]

Next, we claim that the function $E\mapsto \Phi(1;E)$ is an entire function of order $\frac12$ in $E$. The fact that the function is entire follows from the expression \eqref{eq:Z0}, the representation \eqref{eq:transK} and the fact that $K_q(t,s)$, $g(x)$ and $\rho(y)$ are independent of $E$.
To show that the order is $\frac12$, note that
\begin{equation*}
|Z_q(t;E)| \leq \frac{2\Gamma(m+1)t^{m+\frac12}\mathrm{e}^{|\lambda|t}}{\sqrt{\pi}\Gamma\big(m+\frac12\big) }
\left(1+\int_0^t |K_q(t,s)|ds \right) 
\end{equation*}
as a consequence of \eqref{lem:Besselorder1} and \eqref{eq:transK}.
Then,
\[ |\Phi(x;E)| \leq A(x)  \exp(|\lambda|\tau(x)), \]
in which
\[ A(x) = \rho_4(g(x))\frac{2\Gamma(m+1)\tau(x)^{m+\frac12}}{\sqrt{\pi}\Gamma\left(m+\frac{1}{2}\right)}
\left(1+\int_0^{\tau(x)} |K_q(\tau(x),s)|\,\mathrm{d}s \right) \]
is independent of $\lambda$. Since $|\lambda| = O(|E|^{\frac12})$, it follows that $E\mapsto \Phi(x;E)$ has indeed order $\frac{1}{2}$.

Now, we show that the function $E\mapsto \Phi(1;E)$ has purely real zeros, $\tilde{E}_n$, satisfying $\tilde{E}_n \leq -c_1 n^2$ for $n\in\mathbb N$ and suitable constant $c_1>0$.
Indeed, from \eqref{eq:atgprime}, one has $\Phi(1;E)=0$ if and only if $Z_q(b_2;E) = 0$ where 
$b_2=\tau(1)$. In the framework of Lemma~\ref{lemma:5}, $Z_q$ is an eigenfunction of the eigenvalue equation \eqref{eq:zLNF3}. The zeros of $\Phi(1;\cdot)$ are given by $-4m\tilde{E}_n = \lambda_n^2$, in which the $\lambda_n^2$
are the eigenvalues of the self-adjoint problem
\[ -Z''(t) + \frac{m^2-\frac{1}{4}}{t^2} Z(t) + q(t) Z(t) = \lambda^2 Z(t), \qquad \qquad Z(b_2) = 0, \]
with Friedrichs condition at $t=0$ if required (see Remark~\ref{rem:domainZeq}).
The $\tilde{E}_n$ are therefore purely real. They will satisfy the required asymptotics, if the correponding 
$\lambda_n$ are $\geq O(n)$ for large $n$. The latter follows from the fact that $q$ is bounded,
 via a comparison result, removing the term $q(t)Z(t)$ and reducing the question to the corresponding one for Bessel zeros already considered in the proof of Lemma~\ref{lem:speclin}.

We complete the proof as follows. Let 
$
\varphi(\lambda)=\Phi\big(1,-\frac{\lambda^2}{2m}\big).
$
Then,
\[
   \varphi(i\lambda)=\Phi(1;-E)=\Phi(-1;E)
\]
and so the condition for the matching of the boundary values of $\Phi$ is
\[
     \frac{\varphi(i\lambda)}{\varphi(\lambda)}=1.
\] 
Expanding $\Phi(1;E)$ as a Weierstrass product, using Hadamard's Theorem, gives
\[
     \Phi(1;E)=\mathrm{e}^{\gamma(E)}\prod_{n=1}^\infty \left(1-\frac{E}{\tilde{E}_n}\right)
\]
for an entire function $\gamma(E)$, where the convergence of the product is ensured by the fact that $|\tilde{E}_n|\geq c_1n^2$. 
The fact that $\Phi(1;E)$ has order $\frac{1}{2}$ means that $\gamma(E)$ is constant. Thus
\[
   \frac{\varphi(i\lambda)}{\varphi(\lambda)}= \prod_{n=1}^\infty \frac{1-\frac{\lambda^2}{\lambda_n^2}}{1+\frac{\lambda^2}{\lambda_n^2}}.
\]
Since the $\lambda_n^2$ are on the real line and $\arg(\pm \lambda)\in \left(-\frac{\pi}{4},\frac{\pi}{4}\right)$,
\[
    \left|1-\frac{\lambda^2}{\lambda_n^2}\right|<\left|1+\frac{\lambda^2}{\lambda_n^2}\right|
\]
whence
\begin{equation} \label{modquot}
     \left|\frac{\varphi(i\lambda)}{\varphi(\lambda)}\right|<1
\end{equation}
for $\arg(\pm \lambda)\in \left(-\frac{\pi}{4},\frac{\pi}{4}\right)$. These inequalities are swapped in the complementary $\lambda$-region $\arg(\pm \lambda)\in \left(\frac{\pi}{4},\frac{3\pi}{4}\right)$. Hence, 
\[
     \frac{\varphi(i\lambda)}{\varphi(\lambda)}=1
\]
is only possible for $\arg(\pm \lambda)=\pm\frac{\pi}{4}$, which corresponds to purely imaginary $E$.
\end{proof}

\section{Proof of Theorem~\ref{th:thepseudospectrum}\label{section:5}}

The construction of the pseudo-modes in Proposition~\ref{claim:2} consisted of three steps, which we now generalise. In order to complete steps~2 and 3, we shall need replacements for the inequalities which were proved in Section~\ref{section:4} using explicit Bessel function expressions. 
The idea is to consider the solution for general $f$ as a perturbation of these, via the expression of $\Phi(x;E)$ in \eqref{eq:atg} and the integral representation \eqref{eq:transK}.
We perform this task in the first part of this section and complete the proof in the second part.

We refer to Table~\ref{table1} for a summary of the specific relations between the different parameters appearing in the various statements below.

\begin{lemma}\label{lem:zasy}  Let $E\in\mathbb{C}$ be such that $\arg(E)\not\in \{\frac{k\pi}{2}:k\in\mathbb{Z}\}$. The solution of \eqref{eq:zLNF3} satisfies the following estimate.
There exist a constant $R>0$ independent of $q(\cdot)$ and a constant $C_q>0$, such that 
\begin{equation}
\label{eq:zasy}
Z_q(t;E) = Z_0(t;E) \left( 1 + W_q(t;E)\right)
\end{equation}
where  
\begin{equation} \label{lesssazy}
     |W_q(t;E)|\leq \frac{C_q}{|E|^{\frac12}}
\end{equation}
for all $|E|>R$ and $t\in[0,b_2]$.
For each sufficiently large $R$ and sufficiently small $\alpha$, the estimate on $W_q$ holds uniformly in the region 
\begin{equation} \label{setE}
\left\{E\in\mathbb{C}\,:\,|E|>R,\,\inf_{k=0,1,2,3}\left|\arg(E)-\frac{k\pi}{2}\right|\geq \alpha\right\}. 
\end{equation}
\end{lemma}
\begin{proof}
Without loss of generality  we assume that $\arg(E)\in(0,\frac{\pi}{2})$. Recall the correspondence $\lambda^2=-4mE$, for $\lambda=|\lambda|\mathrm{e}^{i(\theta+\frac{\pi}{2})}$ and $\theta\in(0,\frac{\pi}{4})$. Fixing $\theta$ is equivalent to fixing $\arg(E)$.

Let $\delta>0$ be small. We use \eqref{eq:transK}, splitting the integral as
\begin{equation*}
Z_q(t;E) = Z_0(t;E)\left\{1 + \left(\int_{0}^{\min(t,\frac{\delta}{\sqrt{|E|}})} +  \int_{\min(t,\frac{\delta}{\sqrt{|E|}})}^t\right) K_q(t,s) \frac{Z_0(s;E)}{Z_0(t;E)}\,\mathrm{d}s \right\}
\end{equation*}
and we show that both integrals are $O(|\lambda|^{-1})$. 
Note that $0<\frac{s}{t}< 1$ in both integrals.
From \eqref{eq:Z0} we have
\[ \frac{Z_0(s;E)}{Z_0(t;E)} = \left(\frac{s}{t}\right)^{\frac12}\frac{J_m(\lambda s)}{J_m(\lambda t)}. \]

 For the first integral, we have to estimate the quotient $\frac{J_m(\lambda s)}{J_m(\lambda t)}$ either in the triangle \[\mathsf{T}=\left\{(s,t)\in\mathbb{R}^2:0<s<t<\frac{2\delta  m^{\frac12}}{|\lambda|}\right\},\] 
whenever $t< \frac{2\delta m^{\frac12}}{|\lambda|}$,
 or in the strip \[\mathsf{S}=\left\{(s,t)\in\mathbb{R}^2:0<s<\frac{2\delta  m^{\frac12}}{|\lambda|}\right\}\] whenever $t\geq \frac{2\delta m^{\frac12}}{|\lambda|}$.  
Since $\frac{J_m(\alpha)}{J_m(\beta)}$ is bounded uniformly for all $|\alpha|\leq |\beta|\leq 2\delta m^{\frac12}$, provided $\delta$ is chosen sufficiently small to avoid the first zero of $J_m$, we know that \[\sup_{(s,t)\in \mathsf{T}} \left|\frac{J_m(\lambda s)}{J_m(\lambda t)}\right|<c_1\] where the constant $c_1$ is independent of $|\lambda|$. On $\mathsf{S}$, the numerator $J_m(\lambda s)$ is bounded independently of $\lambda$ because
$|\lambda s|$ is bounded, while in the denominator the term $J_m(\lambda t)$ is bounded away from zero by \eqref{preasympBessel} and the fact that $\arg(\lambda t)\in\big(\frac{\pi}{2},\frac{3\pi}{4}\big)$. Note that, by hypothesis, $\lambda t$ lies on a fixed ray $\lambda t = |\lambda t|\exp(i(\theta+\frac{\pi}{2}))$, away from Bessel zeros. The above, together with the property ii), that the sup-norm of the transformator kernel $K_q$ is finite, implies that
\[ \left|\int_{0}^{\frac{\delta}{\sqrt{|E|}}} \hspace{-1mm} K_q(t,s) \frac{Z_0(s;E)}{Z_0(t;E)}\,\mathrm{d}s \right| \leq c_2 |\lambda|^{-1}, \]
for a constant $c_2>0$ depending only upon $\arg(\lambda)$. This gives the estimate for the first integral.

For the second integral, first recall the right hand side of \eqref{asympBessel}. Then, we have
\begin{align*}
\left| \frac{Z_0(s;E)}{Z_0(t;E)} \right| &=\left(\frac{s}{t}\right)^{\frac12} \frac{J_m(\lambda s)}{J_m(\lambda t)} \\ &=
\left(\frac{s}{t}\right)^{\frac12} \frac{|J_m(i|\lambda| s e^{i \theta})|}{|J_m(i|\lambda| t e^{i \theta})|} \\
&\sim \mathrm{e}^{|\lambda| (s-t)\cos \theta}
\end{align*}
where the limit is uniform for $\theta\in[0,\frac{\pi}{4}-\delta]$. Therefore,
there exists a constant $c_3>0$, such that 
\[ \left| \frac{Z_0(s;E)}{Z_0(t;E)} \right| \leq c_3\exp(|\lambda|(s-t)\cos\theta), \]
for all $\theta\in[0,\frac{\pi}{4}-\delta]$ and $|\lambda|>1$.
Hence, the second integral is bounded by
\[ c_4 \int_{\frac{\delta}{\sqrt{|E|}}}^{t} \exp(-|\lambda|(t-s)\cos\theta )\,\mathrm{d}s \leq \frac{c_4}{|\lambda|\cos\theta}, \]
which is $O(|\lambda|^{-1})$ for large $|\lambda|$.

This ensures the existence of $R>0$ and $C_q>0$ satisfying \eqref{eq:zasy}. Note that $c_1$, $c_2$ and $c_4$ depend on $q(\cdot)$, but  we can choose $R>0$ independent of $q(\cdot)$. Also note that all the estimates above are unform with respect to $\theta$ and $t$.   
\end{proof}

By \eqref{eq:zasy}, the solution $\Phi(x;E)$ is such that
\begin{equation} \label{asympeigen}
\begin{aligned}
\Phi(x;E)&=\rho_4(g(x))\frac{2^m\Gamma(m+1)}{\lambda^m}\tau(x)^{\frac12} J_m(\lambda \tau(x)) \big(1+O(|\lambda|^{-1})\big) \\
\Phi(x;-E)&=\rho_4(g(x)) \frac{2^m\Gamma(m+1)}{(i\lambda)^m}\tau(x)^{\frac12} J_m(i\lambda \tau(x))\big(1+O(|\lambda|^{-1})\big)
\end{aligned}
\end{equation}
as $|\lambda|\to\infty$, where the $O(|\lambda|^{-1})$ bounds are uniform for $E=-\frac{\lambda^2}{4m}$ in the region \eqref{setE} of Lemma~\ref{lem:zasy}.

Now we consider estimates leading to upper bounds for the norm of the action of $(L-E)$ on the pseudo-modes constructed below.

\begin{lemma}\label{predPhibound} 
Let $\lambda=|\lambda|\mathrm{e}^{i\left(\frac{\pi}{2}+\theta\right)}$ where $\theta\in\Big(0,\frac{\pi}{4}\Big)$ and $|\lambda|> \sqrt{2}$. Set $E\equiv E_\lambda=-\frac{\lambda^2}{4m}$. Let $0<\delta<\frac{\pi}{8}$ be fixed.
Then, there exists a constant $C_f>0$ only depending on $f(\cdot)$ and $\delta$, such that
\begin{equation} \label{mm:upperbound}
    |\Phi(x;E)|\leq C_f
\end{equation}
and
\begin{equation} \label{boundphiprime} |\Phi'(x;E)| \leq C_f|\lambda|^2, \end{equation}
for all $x\in\left[-\frac{2}{|\lambda|^2},-\frac{1}{|\lambda|^2}\right]$ and $\theta\in[\delta,\frac{\pi}{4}-\delta]$.
\end{lemma}
\begin{proof}
Everywhere in this proof $|\lambda|> \sqrt{2}$ is fixed and none of the constants $c_j$ depends on $|\lambda|$ or $\theta$. Whenever this is not obvious, we will explain the reason for this independence.

First, let us show \eqref{mm:upperbound}. From Lemma~\ref{lem:2} and the definition of $\tau(x)$, we have that $\tau(x)\sim g(x) \sim |x|^{\frac12}$ as $x\to 0$. Also, recall that $\rho_4(y)=O(y^{-(\ell+1)})$ as $y\to 0$. We remark here that these functions are independent of $\theta$. Then, from the second equation in \eqref{asympeigen} it follows that, 
\[
\begin{aligned}
   |\Phi(x;E)|& = \rho_4(g(x)) \left[\frac{2^m\Gamma(m+1)}{|\lambda|^{m}}\right]\tau(x)^{\frac12} \Big|J_m(-|\lambda|\tau(x)\mathrm{e}^{i\theta})\Big|(1+O(|\lambda|^{-1}))\\ 
&=2^m\Gamma(m+1) \rho_4(g(x)) \tau(x)^{m+\frac12} \frac{\big|J_m(-|\lambda|\tau(x)\mathrm{e}^{i\theta})\big|}{(|\lambda|\tau(x))^m}(1+O(|\lambda|^{-1}))\\ 
&\leq c_1 \sup_{|z|\leq \sqrt{2}+\frac12}\left|\frac{J_m(z)}{z^m}\right|(1+O(|\lambda|^{-1})) \\ & \leq c_2,
\end{aligned}
\]
 for all $x\in\left[-\frac{2}{|\lambda|^2},-\frac{1}{|\lambda|^2}\right]$ and $\theta\in[\delta,\frac{\pi}{4}-\delta]$.
Here the constants $c_1>0$ depends only on $f(\cdot)$ and, according to  Lemma~\ref{lem:zasy}, $c_2>0$ can be chosen to depend only on $f(\cdot)$ and $\delta$. This gives \eqref{mm:upperbound} as required in the statement.

Now, consider \eqref{boundphiprime}. The expression \eqref{eq:atgprime} and a straightforward calculation, give 
\begin{align*} \Phi'(x;E) & = -\frac{\rho'(g(x))g'(x)}{4} [1+\rho(g(x))]^{-\frac{5}{4}} \frac{Z_q(\tau(x);E)}{g(x)^{\ell+1}} \\
& -(\ell+1)\frac{g'(x)}{g(x)} [1+\rho(g(x))]^{-\frac14}\frac{Z_q(\tau(x);E)}{g(x)^{\ell+1}} \\
& + \frac{g'(x)}{g(x)}[1+\rho(g(x))]^{\frac14}\frac{Z_q'(\tau(x);E)}{g(x)^{\ell}} \\
&=A(x)+B(x)+D(x).
\end{align*}
We provide estimates for the modulus of these three quantities, that lead to \eqref{boundphiprime}.

\underline{Estimates for $|A(x)|$ and $|B(x)|$}. 
Recall the asymptotics in the regime $x\to 0$: $g(x)\sim |x|^{\frac12}$, $\rho(g(x))=O(x)$ and $\rho'(g(x))g'(x)=O(1)$, where the quantities involved only depend on $f(\cdot)$.
By writing $Z_q(\tau(x);E)$ in terms of $\Phi(x;E)$ using \eqref{eq:atgprime}, and by applying \eqref{mm:upperbound}, we then know that there exists a constant $c_3>0$, only depending on $f(\cdot)$ and $\delta$, such that
\[
    |A(x)|=\frac{|\rho'(g(x))g'(x)|}{4} [1+\rho(g(x))]^{-1} |\Phi(x;E)|\leq c_3
\]
for all $x\in\left[-\frac{2}{|\lambda|^2},-\frac{1}{|\lambda|^2}\right]$ and $\theta\in[\delta,\frac{\pi}{4}-\delta]$.
Similarly, since $\frac{g'(x)}{g(x)} \sim \frac12 x^{-1}$
 as $x\to 0$, there exists a constant $c_4>0$, only depending on $f(\cdot)$ and $\delta$, such that
\[
|B(x)|=(\ell+1)\left|\frac{g'(x)}{g(x)}\right||\Phi(x;E)|\leq c_4|\lambda|^2
\]
 for all $x\in\left[-\frac{2}{|\lambda|^2},-\frac{1}{|\lambda|^2}\right]$ and $\theta\in[\delta,\frac{\pi}{4}-\delta]$.
 
\underline{Estimate for $|D(x)|$}. 
By integrating, we obtain,
\[ Z_q'(\tau(x);E) = \int_0^{\tau(x)} Z_q''(s;E)\,\mathrm{d}s = \int_0^{\tau(x)}\left(\frac{m^2-\frac14}{s^2}-q(s)-\lambda^2\right)Z_q(s;E)\mathrm{d}s. \]
According to Lemma \ref{lem:zasy} and \eqref{eq:Z0}, we have
\begin{equation} \label{triplestar}
\begin{aligned} Z_q(t;E) & = Z_0(t;E)\Big(1+O(|E|^{-\frac12})\Big)  \\ & = 2^m \Gamma(m+1)  t^{m+\frac12}\frac{J_m(\lambda t)}{(\lambda t)^m} \Big(1+O(|\lambda|^{-1})\Big)
\end{aligned}\end{equation}
where,  for fixed $q(\cdot)$, $m$ and $\delta$, the $O(|\lambda|^{-1})$ bound is uniform for all $\theta\in[\delta,\frac{\pi}{4}-\delta]$ as $|\lambda|\to\infty$. 
Since
\[  \frac{|J_m(z)|}{|z|^m}\sim \frac{1}{2^m\Gamma(m+1)}\]
for $|z|\to 0$, then there exists a constant $c_5>0$ such that
\[ \begin{aligned} |Z_q'(\tau(x);E)| 
&\leq c_5 \int_0^{\tau(x)}\left(\frac{m^2-\frac14}{s^2}+|q(s)| + |\lambda|^2\right)s^{m+\frac12}2^m\Gamma(m+1) \frac{|J_m(\lambda s)|}{(|\lambda| s)^m}\mathrm{d}s  \\
&\leq c_5\int_0^{\tau(x)}\left(\frac{m^2-\frac14}{s^2}+\| q \|_\infty + |\lambda|^2\right)s^{m+\frac12}\,\mathrm{d}s\end{aligned} \]
for all $x\in \big[-\frac{2}{|\lambda|^2},-\frac{1}{|\lambda|^2}\big]$ and $\theta\in[\delta,\frac{\pi}{4}-\delta]$.
Moreover, for all such $(x,\theta)$,
\[ \int_0^{\tau(x)} \frac{m^2-\frac14}{s^2}s^{m+\frac12}\mathrm{d}s = \left(m+\frac12\right)\tau(x)^{m-\frac12} \leq c_6|\lambda|^{\frac12-m} \]
 and 
\[ \int_0^{\tau(x)}|\lambda|^2 s^{m+\frac12}\mathrm{d}s = |\lambda|^2 \frac{\tau(x)^{m+\frac{3}{2}}}{m+\frac{3}{2}} \leq
 c_7|\lambda|^{\frac12-m}, \]
 where $c_6,\,c_7>0$ depend only on $m$. Since $m>\frac12$ and $\frac12-m=-\ell$, we then can conclude that there exists a constant $c_8>0$ only depending on $f(\cdot)$ and $\delta$, such that
\[
\frac{|Z_q'(\tau(x);E)|}{g(x)^{\ell}}\leq c_8
\]
for all $x\in\left[-\frac{2}{|\lambda|^2},-\frac{1}{|\lambda|^2}\right]$ and  $\theta\in[\delta,\frac{\pi}{4}-\delta]$. Thus, for a suitable constant $c_9>0$ only depending on $f(\cdot)$ and $\delta$,
\[
\begin{aligned}   
|D(x)|&=\left|\frac{g'(x)}{g(x)}\right|[1+\rho(g(x))]^{\frac14}\left|\frac{|Z_q'(\tau(x);E)|}{g(x)^{\ell}}\right|\leq c_9|\lambda|^2
\end{aligned}
\] 
for all $x\in\left[-\frac{2}{|\lambda|^2},-\frac{1}{|\lambda|^2}\right]$ and  $\theta\in[\delta,\frac{\pi}{4}-\delta]$.

These estimates for $|A(x)|,\,|B(x)|$ and $|D(x)|$, directly imply \eqref{boundphiprime}
and confirm the validity of the lemma.
\end{proof}

We are now ready to complete the proof of Theorem~\ref{th:thepseudospectrum}, which will be a direct corollary of the next proposition. 
Similarly to Proposition~\ref{claim:2}, here we can choose the constant $R>0$ independent of $\lambda$, $\delta$, $\varepsilon$ and $r(\cdot)$. The constant $C_{f}>0$ on the other hand, is independent of $\lambda$ and $\delta$, but depends on $r(\cdot)$ and $\varepsilon$.

\begin{proposition}\label{thm:genclaim2} 
Let $\lambda\in\mathbb{C}$ be given by \eqref{lambdaandE} and set $E\equiv E_\lambda=-\frac{\lambda^2}{4m}$.   Let $0<\delta<\frac{\pi}{8}$ be fixed. There exists a constant $R>0$ and a constant $C_f>0$, depending only upon $f(\cdot)$, such that for all $|\lambda|\geq R$ and $\theta\in[\delta,\frac{\pi}{4}-\delta]$, we can find a pseudo-mode $\quasi\equiv \quasi_{\lambda}\in \operatorname{Dom}(L)$ satisfying the inequality
\begin{equation}  \label{pseudospecasympgen}
    \frac{\|(L-E_{\lambda})\quasi_{\lambda}\|}{\|\quasi_{\lambda}\|}\leq C_f|\lambda|^{m+\frac32} \mathrm{e}^{-\tau(\frac12)|\lambda|\sin \theta}. 
\end{equation} 
\end{proposition}
\begin{proof}
As in the proof of Proposition~\ref{claim:2}, we construct the pseudo-mode $\quasi(x)$ as the product of the periodiser $\chi(x)$ and the regular solution $\Phi(x,E)$. 

\underline{Step 1}: construction of $\quasi$.
Let $\chi(x)$ be given by \eqref{eq:chidef}. Note that in the expression for $\chi(x)$, the
term $\frac{\Phi(-1;E)}{\Phi(1;E)}$ appears. According to \eqref{modquot} in the proof of Theorem~\ref{lem:specgen} we know that $\left|\frac{\Phi(-1;E)}{\Phi(1;E)}\right|<1$ for all $E=E_{\lambda}$. So the bounds \eqref{star} on the derivatives of $\chi(x)$ continue to be valid.

\underline{Step 2}: lower bound on the denominator of \eqref{pseudospecasympgen}. This extends \eqref{lowerdenominator}.
We show that, for a suitable constant $c_1>0$ and sufficiently large $R>0$, we have
\begin{equation}\label{eq:newlowerdenominator}
 \| \quasi \| \geq c_1 \frac{2^{m} \Gamma(m+1)}{|\lambda|^{m+\frac12}} \mathrm{e}^{\tau(\frac12)|\lambda| \sin \theta}
\end{equation}
for all $|\lambda|\geq R$. 
For this purpose, we first show that
\begin{equation} \label{eq:step2}
     |\Phi(-x;E)|\geq c_1 \frac{2^{m} \Gamma(m+1)}{|\lambda|^{m+\frac12}} \mathrm{e}^{\tau(\frac12)|\lambda| \sin \theta}
\end{equation}
for all $x\in[\frac12,1]$ and $|\lambda|\geq R$. By virtue of the symmetry $\mathcal{P}$, \eqref{eq:atgprime}, Lemma~\ref{lem:zasy}  and \eqref{eq:Z0}, we have
\[ \begin{aligned}
|\Phi(-x;E)| &= |\Phi(x;-E)| \\
&= \rho_4(g(x)) Z_q(\tau(x);E) \\
&= \rho_4(g(x)) Z_0(\tau(x);E)\left(1+W_q(\tau(x);E)\right) 
\\&= \rho_4(g(x))\frac{2^m\Gamma(m+1)}{|\lambda|^m}\tau(x)^{\frac12}J_m(-|\lambda| \tau(x)\mathrm{e}^{i\theta})\left(1+W_q(\tau(x);E)\right) \end{aligned} \]
in which $x>0$ and $|W_q(\tau(x);E)|=O(|\lambda|^{-1})$ satisfies \eqref{lesssazy}. Hence, from the left hand side of  \eqref{asympBessel}, 
\[ |\Phi(-x;E)| \sim |\rho_4(g(x))|\frac{2^m\Gamma(m+1)}{\sqrt{2\pi}|\lambda|^{m+\frac12}}\mathrm{e}^{|\lambda|\tau(x)\sin\theta} \]
for $|E|\to \infty$. From the fact that $g(x)$ is increasing (Lemma~\ref{lem:2}) and $\mathrm{C}^3((0,1])$, it follows that there exists $c_2>0$ such that \[|\rho_4(g(x))|=\left|(1+\rho(g(x)))^{-\frac14}g(x)^{-(m+\frac12)}\right|\geq c_2\] for all $x\in [\frac12,1]$. Therefore, since
$\tau(x)$ is increasing with $x$, then \eqref{eq:step2} follows. Here, note that $c_2$ is independent of $\varepsilon$; 
also that $R>0$ can be chosen independent of $r(\cdot)$ or $\varepsilon$, by making $c_1>0$ dependent on $q(\cdot)$ following Lemma~\ref{lem:zasy}.  By integrating $|\quasi(x)|^2$  over $[-1,-\frac12]$ and taking square roots,
the proof of \eqref{eq:newlowerdenominator} follows.

\underline{Step 3}: upper bound on the numerator of \eqref{pseudospecasympgen}. We complete the proof, by showing that there exists a constant $c_3>0$ such that for all $|\lambda|>\sqrt{2}$, 
\begin{equation} \label{upperProp2} \|(L-E)\quasi\| \leq c_3 |\lambda|. \end{equation} We aim for an upper bound analogous to \eqref{uppernumerator} in the proof of Proposition~\ref{claim:2}. Recall that
\[ (L-E)\quasi = \Phi [L\chi] + 2f\chi'\Phi'. \]
Hence
\begin{align*}
\|(L-E)\quasi\|&\leq \|\Phi L[\chi]\|+\|2f\chi' \Phi' \| \\
     &\leq \frac{1}{|\lambda|}\Big(\|\Phi L[\chi]\|_{\infty}+ 2   \| f\chi' \Phi'\|_{\infty} \Big)
\end{align*}

Consider the second term.
We know that
\[ \sup_{x\in \left[\frac{1}{|\lambda|^2}, \frac{2}{|\lambda|^2}\right]}|f(x)| \leq \frac{2\| f' \|_{\infty}}{|\lambda|^2} \]
for $|\lambda|>\sqrt{2}$. From the bound for $|\chi'(x)|$ in \eqref{star} and from \eqref{boundphiprime}, we then have
\begin{equation} \label{est1} 2\|f\chi'\Phi'\|_\infty    \leq c_4|\lambda|^2 \end{equation}
for $|\lambda|>\sqrt{2}$.

Now, consider the first term. The bounds for $|\chi'(x)|$ and $|\chi''(x)|$ from \eqref{star}, and the fact that their supports lie in $\Big[-\frac{2}{|\lambda|^2},-\frac{1}{|\lambda|^2}\Big]$, give 
\begin{equation}\label{lchibound}
    |L[\chi](x)|=\big|f(x) \chi''(x)+(f'(x)+1)\chi'(x)\big|\leq c_5 |\lambda|^2
\end{equation}
for a constant $c_5>0$ only dependant on $r(\cdot)$. 
According to \eqref{mm:upperbound}, 
\[
    \sup_{\left[-\frac{2}{|\lambda|^2},-\frac{1}{|\lambda|^2}\right]}|\Phi(x;E)|\leq c_{6}.
\]
Then, 
\begin{equation} \label{est2}
   \|\Phi L[\chi]\|_{\infty}\leq c_{7}|\lambda|^2
\end{equation}
for $|\lambda|>\sqrt{2}$.

Estimates \eqref{est1} and \eqref{est2} imply \eqref{upperProp2}. 
\end{proof}

\section{Proof of Theorem~\ref{th:upperboundpseudospectrum}} \label{sec6}
In this final section we show that the uper bound given in Theorem~\ref{th:upperboundpseudospectrum} is valid, by first showing that the resolvent of $L$ is in a suitable Schatten class. The next lemma follows similar results established in \cite[Proposition~4.3]{ChugKarabas} and \cite[Theorem~8]{BMR2012}. We have not found any evidence suggesting that the threshold $\frac23$ in the lemma is not optimal. 

\begin{lemma} \label{LemmaSchatten}
For all $p>\frac23$ and $E\not\in\operatorname{Spec}(L)$, we have
$(L-E)^{-1}\in\mathcal{C}_{p}$.
\end{lemma}
\begin{proof}
Consider the reduced operator 
$\tilde{L}:\operatorname{Dom}(\tilde{L})\longrightarrow \operatorname{Ran}(L)$ introduced in the proof of Lemma~\ref{Lemma1} at the end of Section~\ref{sec2}. Since both, domain and range, are closed subspaces of co-dimension 1, it is enough to show that the inverse, $\tilde{S}$, lies in the required Schatten class. Now, $\tilde{S}$ is the compression to $\operatorname{Ran}(L)$, of a rank-one perturbation of the integral operator $S:\mathrm{L}^2(-1,1)\longrightarrow \mathrm{L}^2(-1,1)$,  
\[
     Sw(z)=\int_{0}^z H(x,z)w(z)\,\mathrm{d}z,
\]   
where $H(x,z)$ is the kernel from the expression \eqref{kernelH}. Therefore, it is enough to show that $S\in\mathcal{C}_{p}$ for $p>\frac23$.  We split the proof of this into four steps.

\underline{Step 1}. Note that
\begin{align*}
    \int_{-1}^1 H(x,z)w(z)\,\mathrm{d}z&=\int_{-1}^1 H(x,z)\left(\frac{\mathrm{d}}{\mathrm{d}z}\int_0^z w(t)\,\mathrm{d}t\right)\,\mathrm{d}z \\
    &= \left[H(x,z)\int_0^z w(t)\,\mathrm{d}t\right]_{z=-1}^{z=1}-\int_{-1}^1 \partial_z H(x,z)\int_0^z w(t)\,\mathrm{d}t\,\mathrm{d}z \\
    &=\int_{-1}^1 \tilde{H}(x,z) \left(\int_{0}^{z}|z|^{-\gamma}w(t)\,\mathrm{d}t\right) \,\mathrm{d}z,
\end{align*}
where $\gamma \in(0,\frac12)$ and
\[
    \tilde{H}(x,z)=\operatorname{sgn}(x)\frac{|z|^{\gamma} g^{\frac{1}{\varepsilon}}(z)}{g^{\frac{1}{\varepsilon}}(x)f(z)}\mathds{1}_{[0,|x|)}(\operatorname{sgn}(x)z).
\]
Therefore, $S=TV$ where $T$ is the integral operator associated with the kernel $\tilde{H}(x,z)$ and $V$ the integral operator associated with the kernel
\[
    v(z,t)=\operatorname{sgn}(z)|z|^{-\gamma}\mathds{1}_{[0,|z|)}(\operatorname{sgn}(z)t).
\]

\underline{Step 2}. We show that $T\in\mathcal{C}_2$. Indeed,
\[
     \int_{-1}^1 \int_{-1}^1|\tilde{H}(x,z)|^2\mathrm{d}z\mathrm{d}x=
     \int_{-1}^1\frac{1}{g^{\frac{2}{\varepsilon}}(x)}\left( \operatorname{sgn}(x)\int_{0}^x|z|^{2\gamma}\left|\frac{g^{\frac{2}{\varepsilon}}(z)}{f^2(z)}\right|\mathrm{d}z\right) \mathrm{d}x.
\]
The only singularity of the integrand occurs at the origin.
Now, since
\[
    |z|^{2\gamma}\left|\frac{g^{\frac{2}{\varepsilon}}(z)}{f^2(z)}\right|=O(|z|^{2\gamma+\frac{1}{\varepsilon}-2}),
\]
as $|z|\to 0$,
then
\[
   \int_{0}^x|z|^{2\gamma}\left|\frac{g^{\frac{2}{\varepsilon}}(z)}{f^2(z)}\right|\mathrm{d}z=O(|x|^{2\gamma+\frac{1}{\varepsilon}-1})
\]
as $|x|\to 0$. Then 
\[
   \int_{-1}^1 \int_{-1}^1|\tilde{H}(x,z)|^2\mathrm{d}z\mathrm{d}x=\int_{-1}^1
   \tilde{h}(x)\,\mathrm{d}x
\]
where the function $\tilde{h}(x)$ is continuous in $[-1,0)\cup(0,1]$ and $|\tilde{h}(x)|=O(|x|^{2\gamma-1})$ as $|x|\to 0$. Therefore, the double integral is finite, so $T\in\mathcal{C}_2$. 

\underline{Step 3}. We now show that $V\in\mathcal{C}_{\frac{1}{1-\gamma}}$. For this, see \cite[p.1117]{DS1963}, it is enough to show that the kernel $v(z,t)$ satisfies the H\"older condition,
\begin{equation} \label{Holder}
     \delta^{\gamma-\frac12}\left[\int_{-1}^1|v(z,t+\delta)-v(z,t)|^2 \,\mathrm{d}z \right]^{\frac12}\leq \Gamma
\end{equation}
for all $-1\leq t\leq 1$ and $\delta>0$ small enough. To show this, note that the square of the left hand side is bounded by
\[
    \int_0^1\frac{\big|v(|z|,|t|+\delta)-v(|z|,|t|)\big|^2}{\delta^{1-2\gamma}}\mathrm{d}z=\frac{1}{\delta^{1-2\gamma}}\int_{|t|}^{|t|+\delta}\frac{\mathrm{d}z}{|z|^{2\gamma}}=\frac{(|t|+\delta)^{1-2\gamma}-|t|^{1-2\gamma}}{(1-2\gamma)\delta^{1-2\gamma}}.
\]
But since the function $y\mapsto y^{1-2\gamma}$ is H\"older continuous of order $1-2\gamma$ for $y\in(0,1)$, indeed, there exists $\Gamma>0$ such that 
\eqref{Holder} holds true. Then, $V\in\mathcal{C}_{\frac{1}{1-\gamma}}$.

\underline{Step 4}. We conclude the proof of the lemma as follows. The previous two steps and the interpolation inequality of the Schatten classes \cite[Lemma~XI.9.9]{DS1963}, give
\[
      \|S\|_{\mathcal{C}_p}\leq 2^{\frac{1}{p}}\|T\|_{\mathcal{C}_2}\|V\|_{\mathcal{C}_{\frac{1}{1-\gamma}}}
\]
for $\frac{1}{p}=\frac{1}{2}+1-\gamma=\frac{3}{2}-\gamma$. But, since $\gamma$ here can be made arbitrarily close to $0$, we have that, indeed, $p>\frac23$ can be made arbitrarily close to $\frac{2}{3}$. 
\end{proof}

The Carleman-type inequality we apply next was recently implemented in \cite{AS}, in order to derive optimal asymptotic estimates for the resolvent norm of complex potential Schr\"odinger operators.    

\begin{proof}[Proof of Theorem~\ref{th:upperboundpseudospectrum}]
Let 
\[
    \eta=\frac12 \min \big\{|\lambda|:\lambda \in \operatorname{Spec}L\setminus \{0\}\big\}>0.
\]
Without loss of generality, we will prove that the inequality in the conclusion of the theorem,
\[
    \|(L-|E|\mathrm{e}^{i\alpha})^{-1}\|< c\exp\left[a|E|^{p}\right]
\]
holds true for all $|E|\geq 2\eta$. If it happens that $1<2\eta<\infty$, the same conclusion holds for $1\leq |E|\leq 2\eta$, by continuity and compactness.
If $\eta=\infty$ (which has not been ruled out in our results above), we can take $\eta=1$ in the next steps of this proof. Moreover, without loss of generality we will assume that the spectrum of $L$ is unbounded. Otherwise, the proof follows simpler arguments.

Our goal is to apply the next remarkable estimate from \cite[(1.2)]{Bandtlow2004}. If an operator $B\in \mathcal{C}_p$, then 
\begin{equation} \label{Carleman}
      \|(B-z)^{-1}\|\leq \frac{1}{\mathrm{dist}(z,\operatorname{Spec}B)}\exp\left[\frac{a_1}{\mathrm{dist}^p(z,\operatorname{Spec}B)}+d_1\right]
\end{equation}
for suitable constants $a_1>0$ and $d_1>0$, independent of $z$.
Here, we let $B=(L-i\eta)^{-1}\in\mathcal{C}_{p}$. By Theorem~\ref{lem:specgen}, we know that 
\[
     \operatorname{Spec}B=\left\{\frac{1}{\lambda-i\eta}:\lambda\in\operatorname{Spec}L\right\}\cup\{0\}
\]
is purely imaginary. 

First, let us show that there exist constants $c_3>c_2>0$, which can be chosen uniformly in $\alpha$ on compact sets, such that 
\begin{equation} \label{distancesto0}
    \frac{c_2}{|E|}\leq \mathrm{dist}\left(\frac{1}{E-i\eta},\operatorname{Spec} B\right)\leq \frac{c_3}{|E|}
\end{equation}
for $|E|\geq 2\eta$. Let $\tilde{E}=E-i\eta=|\tilde{E}|\mathrm{e}^{i\tilde{\alpha}}$. Then $\mathrm{e}^{i\tilde{\alpha}}\not\in i\mathbb{R}$,
\[
     \frac{|E|}{2}<|\tilde{E}|<\frac{3|E|}{2}
\]
and
\[
\operatorname{dist}\left(\frac{1}{E-i\eta},\operatorname{Spec}B\right)= \inf_{\lambda \in \operatorname{Spec}L}\left|\frac{1}{\tilde{E}}-\frac{1}{\lambda-i\eta}\right| = \frac{1}{|\tilde{E}|}\inf_{\lambda \in \operatorname{Spec}L}\left|1-\frac{|\tilde{E}|\mathrm{e}^{i\tilde{\alpha}}}{\lambda-i\eta}\right|.
\]
We estimate bounds for this infimum. On the one hand,
\[
\inf_{\lambda \in \operatorname{Spec}L}\left|1-\frac{|\tilde{E}|\mathrm{e}^{i\tilde{\alpha}}}{\lambda-i\eta}\right| \leq \liminf_{\substack{|\lambda|\to \infty \\\lambda \in \operatorname{Spec}L}} \left| 1-\frac{|\tilde{E}|\mathrm{e}^{i\tilde{\alpha}}}{\lambda-i\eta}  \right|= 1.
\]
Thus, the second inequality in \eqref{distancesto0} holds true for $c_3=2$.  On the other hand, 
since $\lambda \in \operatorname{Spec}L$ implies $\lambda \in i(-\infty,-2\eta]\cup i[2\eta,\infty)\cup \{ 0 \},$
\[
\inf_{\lambda \in \operatorname{Spec}L}\left|1-\frac{|\tilde{E}|\mathrm{e}^{i\tilde{\alpha}}}{\lambda-i\eta}\right|\geq \inf_{w\in (-\infty,-\eta]\cup[\eta,\infty) }\left|1-\frac{i|\tilde{E}|\mathrm{e}^{i\tilde{\alpha}}}{w}\right|\geq 
\mathrm{dist}(1,\{\mu ie^{i\tilde{\alpha}}: \mu\in\mathbb{R}\}).
\]
Since $i\mathrm{e}^{i\tilde{\alpha}}\not\in \mathbb{R}$, then the right hand side is positive, and there exists $c_2>0$ such that the first inequality in \eqref{distancesto0} is also valid. 

We complete the proof of the theorem as follows. Since, 
\[
    (L-E)^{-1}=B(I-(E-i\eta)B)^{-1}=\frac{1}{\tilde{E}}B\big((E-i\eta)^{-1}-B\big)^{-1},
\]
then 
\[
     \|(L-E)^{-1}\|\leq \frac{2}{|E|}\left\|\big(B-(E-i\eta)^{-1}\big)^{-1}\right\|
\]
for all $|E|\geq 2\eta$. Hence, substituting $z=\frac{1}{E-i\eta}$ and \eqref{distancesto0} into \eqref{Carleman}, yields
\[
    \|(L-E)^{-1}\|< \exp(a_2|E|^{p}+d_1)
\]
ensuring the validity of Theorem~\ref{th:upperboundpseudospectrum}. 
\end{proof}

\appendix

\section{The shape of the pseudo-modes as $|E|$ increases} \label{ApB}

\begin{figure}
 $|\lambda|=10$\includegraphics[width=.45\textwidth]{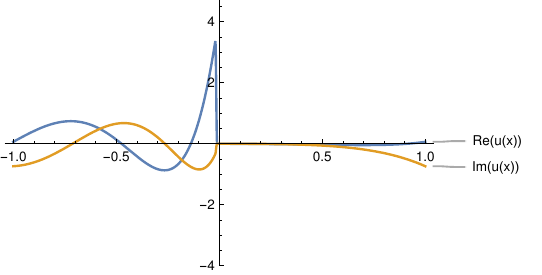} 
\includegraphics[width=.4\textwidth]{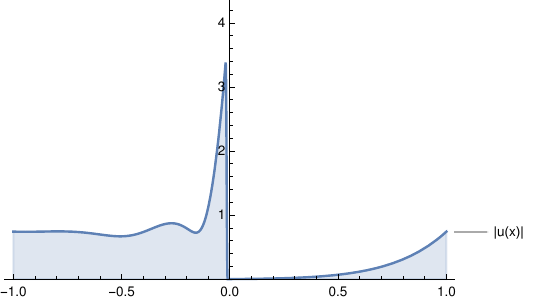}

$|\lambda|=25$ \includegraphics[width=.45\textwidth]{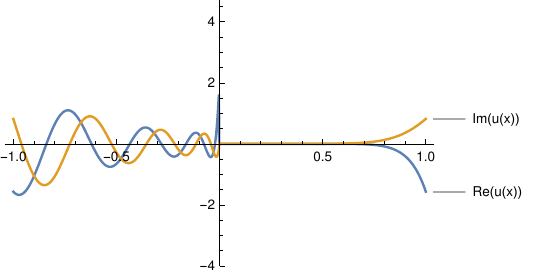} 
\includegraphics[width=.4\textwidth]{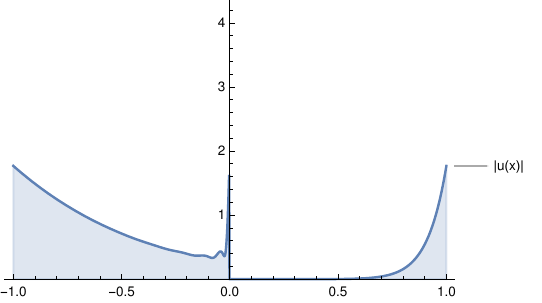}

$|\lambda|=50$ \includegraphics[width=.45\textwidth]{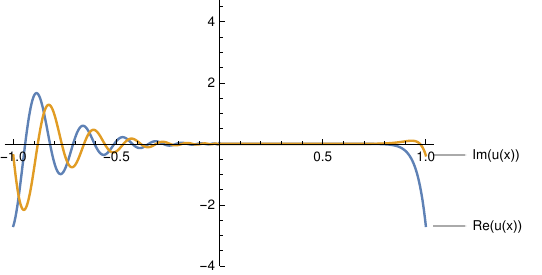} 
\includegraphics[width=.4\textwidth]{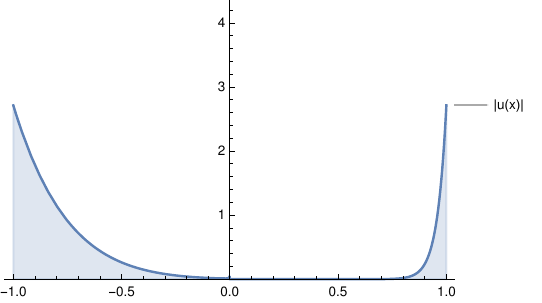}

$|\lambda|=75$ \includegraphics[width=.45\textwidth]{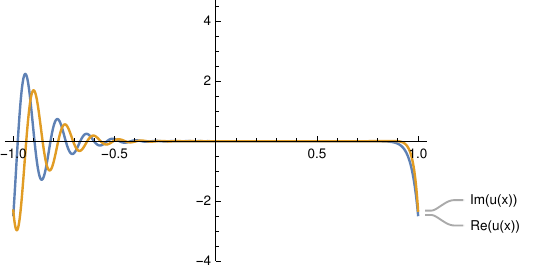}
\includegraphics[width=.4\textwidth]{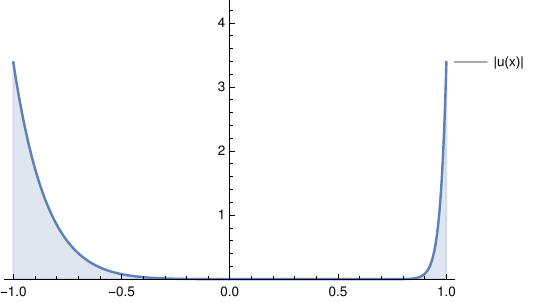}

$|\lambda|=100$ \includegraphics[width=.45\textwidth]{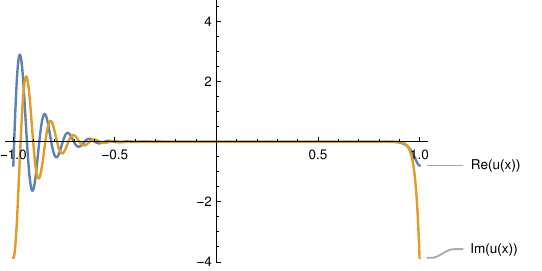} 
\includegraphics[width=.4\textwidth]{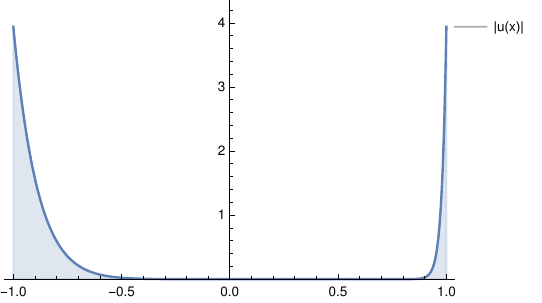}

\caption{Here $f(x)=\frac{x}{2}$ and $\lambda=|\lambda|e^{i\frac{9\pi}{16}}$. We show the pseudo-mode $\mathrm{u}_{\lambda}(x)$ from Proposition~\ref{claim:2} normalised to $\|\mathrm{u}_\lambda\|=1$ for $|\lambda|$ increasing.  }\label{fig3}
\end{figure}

\begin{figure}
 $\theta=\frac{\pi}{32}$
\includegraphics[width=.4\textwidth]{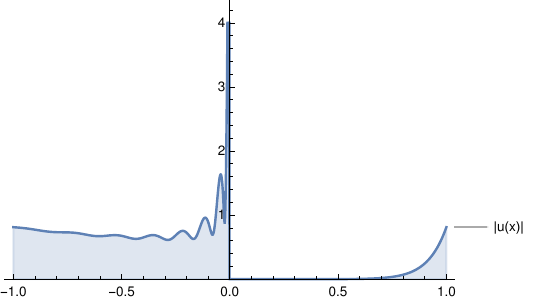} 
\includegraphics[width=.4\textwidth]{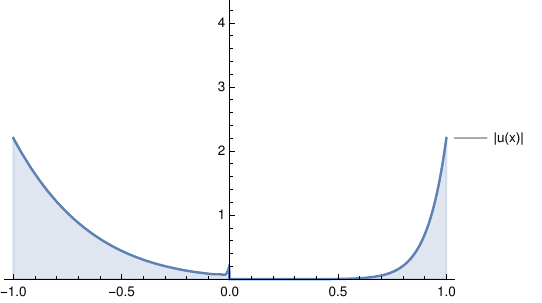} $\theta=\frac{3\pi}{32}$

 $\theta=\frac{\pi}{8}$\includegraphics[width=.4\textwidth]{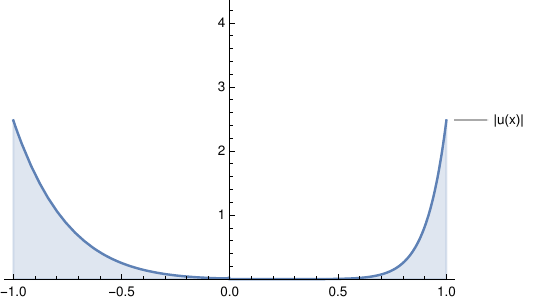} 
\includegraphics[width=.4\textwidth]{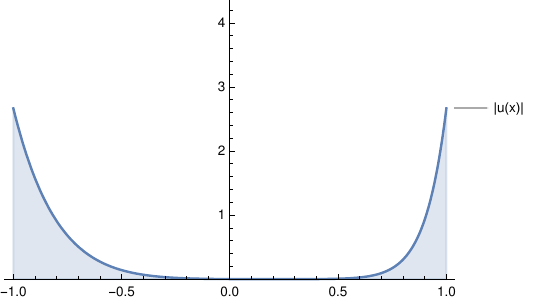} $\theta=\frac{5\pi}{32}$

 $\theta=\frac{3\pi}{16}$\includegraphics[width=.4\textwidth]{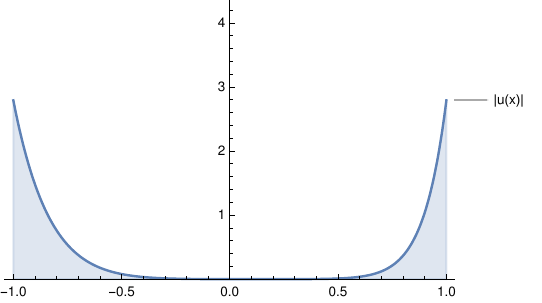} 
\includegraphics[width=.4\textwidth]{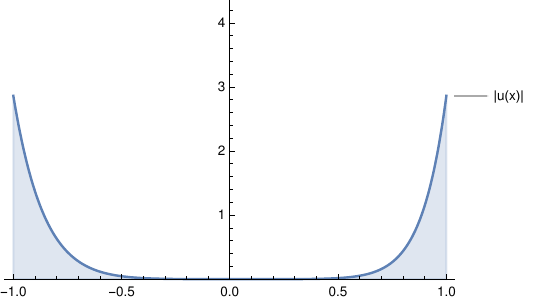} $\theta=\frac{7\pi}{32}$

\caption{Here $f(x)=\frac{x}{2}$ and $\lambda=25e^{i(\frac{\pi}{2}+\theta)}$ for $\theta$ increasing. We show the modulus of the pseudo-mode $|\mathrm{u}_{\lambda}(x)|$ from Proposition~\ref{claim:2} normalised to $\|\mathrm{u}_\lambda\|=1$. For $\theta=\frac{\pi}{16}$ see Figure~\ref{fig3}. }\label{fig4}
\end{figure}

The purpose of this appendix is to illustrate the evolution of the shape of the pseudo-modes constructed in this paper. We consider those in the proof of Proposition~\ref{claim:2}. Below we show two figures with graphs of $\mathrm{u}_{\lambda}(x)$ and $|\mathrm{u}_{\lambda}(x)|$, normalised by $\|\mathrm{u}_{\lambda}\|=1$, for different values of $\lambda$. We have produce these figures by plotting on a computer the exact formula \[\mathrm{u}_{\lambda}(x)=\frac{1}{\|\chi\Phi\|}\chi(x)\Phi(x)\]
where $\Phi(x)$ is the expression \eqref{linearPhi} and $\chi(x)$ is as in \eqref{eq:chidef}. 

In Step~2 of the proof of Proposition~\ref{claim:2}, we have used the crucial fact that the total mass of $\mathrm{u}_\lambda$ in $[-1,-\frac12]$ is exponentially large compared to the norm $\|(L-E)\mathrm{u}_{\lambda}\|$. See the proof of \eqref{lowerdenominator} and the bound \eqref{uppernumerator} as $|\lambda|$ increases. Figure~\ref{fig3} shows this phenomenon in action. As $|\lambda|$ increases from 10 to 100, the pseudomode has a transitional phase from being concentrated near the origin ($|\lambda|=10,\,25$), to accumulating most of the mass in $[-1,-\frac12]$  ($|\lambda|=50,\,75,\,100$). 

We also see in the same figure (left) that the quasi-mode developes an oscillatory behaviour of its real and imaginary parts. Remarkably (right) this oscillatory behaviour completely dissapears in the modulus $|\mathrm{u}_{\lambda}(x)|$. This phenomenon has been observed numerically in the semi-classical regime for related operators, c.f. \cite{Trefethen2005}.

In Figure~\ref{fig4} we fix $|\lambda|=25$ and change $\theta\in(0,\frac{\pi}{4})$ in the expression $\lambda=|\lambda|\mathrm{e}^{i(\frac{\pi}{2}+\theta)}$. When $\theta$ is close to $0$, we see that the mass of the pseudo-mode is mostly accumulated at the origin (this mass will migrate to $[-1,-\frac12]$ eventually as $|\lambda|$ increases). This corresponds to $E$ near the real axis. As $\theta$ increases, $E$ 
gets closer to the imaginary axis, where the spectrum lies, and we see that $|\mathrm{u}_{\lambda}(x)|$ now concentrates towards $\pm1$ with most of the mass in $[-1,-\frac12]\cup[\frac12,1]$. 

\section{Notation used in the paper}

{\small \begin{table}[h]
\begin{tabular}{|c|l|p{5.8cm}|}
\hline
Identities & Parameter  & Constraint or definition \\
\hline
& $x$ & $-1\leq x\leq 1$ \\   &
$\varepsilon$ & $0<\varepsilon<1$\\
$(f(x)\phi'(x)+\phi(x))'=E\phi(x)$ &$f(x)$ & Twice continuously differentiable, odd, analytic at $x=0$, $\sgn(f(x))=\sgn(x)$.\\
 \eqref{eq:1}, \eqref{master_eq_g}, \eqref{eq:bvp}, \eqref{eq:atg}
& & \quad $f(x)=\frac{2\varepsilon x}{1 + x r(x)} =2\varepsilon x+O(x^3)$ \\
&$r(x)$ & Odd, analytic at $x=0$. \\
&$E$ & $E\in\mathbb{C}$ \\
&$\phi(x)$ & Solution, no prescription of BC. \\
&$\Phi(x;E)$ & Scaled solution,  $\Phi(0;E)=1$. \\
\hline
 & $y$ & $y=g(x)$, $0\leq y \leq b_1$ \\ 
$(\varepsilon y F'(y))'+F'(y)=Eh(y)F(y)$ &$b_1$ & $b_1=g(1)$ \\
&$g(x)$ & $f(x)g'(x)=\varepsilon g(x)$, \qquad $g(x)\sim |x|^{\frac12}$. \\ 
\eqref{master_eq_g}, \eqref{superBessel}, \eqref{superBessel2} &$F(y)$ & $\phi(x)=F(g(x))$ \\
&$h(y)$ & $h(y)=\frac{f(g^{-1}(y))}{\varepsilon y}=2y+O(y^3)$ \\ \hline 
& $G(y)$ & $G(y)=y^{\beta}F(y)$ \\$G''(y)-\frac{m^2-\frac14}{y^2}G(y)=$ & $\beta$ & $\beta=\frac{2}{\varepsilon}+\frac12$ \\  \qquad \qquad \qquad$4mE (1+\rho(y))G(y)$ 
&$\lambda$, $\theta$ & $\lambda^2=-4mE$  \\
& &  $\lambda=|\lambda|\mathrm{e}^{i\left(\frac{\pi}{2}+\theta\right)}$ for $0<\theta <\frac{\pi}{4}$ \\
&$m$ & $m=\frac{1}{2\varepsilon}$  \\
\eqref{superBessel0}, \eqref{superBessel3}  &$\rho(y)$ & $\rho(y)=\frac{h(y)}{y}-2=O(y^2)$  \\
&$\rho_3(y)$ & $\rho_3(y)=y\rho(y)$  \\ \hline \eqref{eq:atgprime} & $\rho_4(y)$ & $\rho_4(y)=(1+\rho(y))^{-\frac14}y^{-(m+\frac12)}\sim y^{-(m+\frac12)}$ \\ \hline
& $t$ & $0<t<b_2$, $t=\int_0^y \sqrt{1+\rho(s)}\,\mathrm{d}s$  \\ $-Z''(t) + q(t) Z(t) + $ & $b_2$ & $b_2=\int_0^{b_1} \sqrt{1+\rho(s)}\,\mathrm{d}s$  \\
\qquad \qquad $\frac{\ell(\ell+1)}{t^2} Z(t) = \lambda^2 Z(t)$&$Z(t;E)$ & $Z(t;E)=(1+\rho(y))^{\frac14} G(y)$  \\
& $\ell$ & $\ell=\frac{1}{2\varepsilon}-\frac12=m-\frac12$  \\
\eqref{eq:zLNF3}, \eqref{eq:zLNF}, \eqref{eq:transK} &$q(t)$ & $q:[0,b_2]\longrightarrow \mathbb{R}$ is continuous and bounded. See Lemma~\ref{lemma:5}.\\
&$\tau(x)$ & $\tau(x)=\int_0^{g(x)} \sqrt{1+\rho(s)}\,\mathrm{d}s\sim|x|^{\frac12}$ \\
\hline
\end{tabular} 
\caption{ Relation between the notation and changes of variables used through the paper.}
\label{table1}
\end{table}} 


\section*{Acknowledgements}
We are sincerely grateful to Antonio Arnal, David Krej{\v c}i{\v r}{\'\i}k and Petr Siegl, for their insightful comments and attentive reading of an earlier version of this manuscript. We also thank our colleagues at the Czech Technical University in Prague for hosting many of the discussions that eventually lead to this paper.

\section*{Funding}
M. Marletta was partially funded by EPSRC grant EP/T000902/1. 


\bibliographystyle{ieee}

\end{document}